\documentclass{amsart}

\usepackage{epsfig}
\usepackage{amsmath}
\usepackage{amssymb}
\usepackage[all]{xy}

\def\comment#1{{\sf{[#1]}}}

\def\Z{{\mathbb Z}}
\def\Q{{\mathbb Q}}

\def\R{{\mathbb R}}
\def\C{{\mathbb C}}

\def\bA{{\mathbb A}}

\def\F{{\mathbb F}}
\def\H{{\mathbb H}}
\def\I{{\mathbb I}}

\def\bO{{\mathbb O}}

\def\V{{\mathbb V}}

\def\A{{\mathcal A}}
\def\cD{{\mathcal D}}

\def\M{{\mathcal M}}

\def\cO{{\mathcal O}}
\def\U{{\mathcal U}}
\def\cC{{\mathcal C}}
\def\cG{{\mathcal G}}
\def\cH{{\mathcal H}}

\def\cR{{\mathcal R}}

\def\X{{\mathcal X}}

\def\g{{\mathfrak g}}
\def\h{{\mathfrak h}}

\def\m{{\mathfrak m}}

\def\r{{\mathfrak r}}

\def\u{{\mathfrak u}}

\def\B{\mathfrak B}

\def\e{{\epsilon}}
\def\w{{\omega}}

\def\D{{\Delta}}
\def\G{{\Gamma}}

\def\Ghat{{\widehat{G}}}
\def\Gbar{{\overline{G}}}

\def\Xbar{{\overline{X}}}
\def\Mbar{\overline{\M}}

\def\vv{{\vec{v}}}

\def\aa{{\mathbf a}}
\def\bb{{\mathbf b}}
\def\be{{\mathbf e}}

\def\bZ{{\mathbf Z}}
\def\bG{{\mathbf G}}

\def\bw{\mathbf{w}}

\def\br{{\sf r}}

\def\bB{{A}}

\def\Vec{{\mathsf{Vec}}}
\def\Rep{{\mathsf{Rep}}}
\def\MHS{{\mathsf{MHS}}}
\def\VMHS{{\mathsf{VMHS}}}

\def\HRep{{\sf HRep}}

\def\sl{\mathfrak{sl}}

\def\SL{{\mathrm{SL}}}

\def\Ga{{\mathbb{G}_a}}

\def\betti{{\mathrm{B}}}
\def\DR{{\mathrm{DR}}}

\def\cusp{\mathrm{cusp}}
\def\eis{{\mathrm{eis}}}

\def\op{\mathrm{op}}

\def\fin{\mathrm{fin}}

\def\Fr{\mathcal{F}}
\def\Frbar{{\overline{\Fr}}}

\def\an{\mathrm{an}}

\def\cts{\mathrm{cts}}
\def\fte{\mathrm{f}}

\def\cone{\mathrm{cone}}

\def\dot{{\bullet}}

\def\blank{\phantom{x}}

\def\fss{{/\!/}}

\newcommand\im{\operatorname{im}} 
\newcommand\id{\operatorname{id}}

\newcommand\Spec{\operatorname{Spec}}

\newcommand\Hom{\operatorname{Hom}}

\newcommand\Ext{\operatorname{Ext}}
\newcommand\End{\operatorname{End}}
\newcommand\Aut{\operatorname{Aut}}

\newcommand\Gr{\operatorname{Gr}}

\newcommand\Isom{\operatorname{Isom}}

\newcommand\Mor{\operatorname{Mor}}
\newcommand\Dec{\operatorname{Dec}}
\newcommand\tot{\operatorname{tot}}

\renewcommand\Im{\operatorname{Im}}

\numberwithin{equation}{section}

\newtheorem{theorem}{Theorem}[section]
\newtheorem{lemma}[theorem]{Lemma}
\newtheorem{proposition}[theorem]{Proposition}
\newtheorem{corollary}[theorem]{Corollary}
\newtheorem{bigtheorem}{Theorem}
\newtheorem{bigcorollary}[bigtheorem]{Corollary}

\theoremstyle{definition}
\newtheorem{definition}[theorem]{Definition}

\theoremstyle{remark}
\newtheorem{remark}[theorem]{Remark}

\begin{document}

\title{Deligne--Beilinson Cohomology of Affine Groups}

\dedicatory{To Steven Zucker on the occasion of his 65th birthday}

\author{Richard Hain}
\address{Department of Mathematics\\ Duke University\\
Durham, NC 27708-0320}
\email{hain@math.duke.edu}

\thanks{Supported in part by grant DMS-1406420 from the National Science
Foundation and by the Friends of the Institute for Advance Study.}

\date{\today}

%\subjclass{Primary xxxxx; Secondary xxxxx}

%\keywords{}

%\begin{abstract}
%%\comment{To be added}
%\end{abstract}

\maketitle

\tableofcontents

\section{Introduction}

The purpose of these notes is to develop the basic theory of Deligne--Beilinson
cohomology of affine group schemes that are endowed with a mixed Hodge
structure. These results are needed in the works \cite{hain:modular} on the
Hodge theory of modular groups and \cite{mem} on universal mixed elliptic
motives. The results here generalize results in the unipotent case, which were
established in \cite{carlson-hain}. One motivation for this work is to interpret
Brown's computation \cite{brown:mmv} of the periods of iterated iterated
integrals of Eisenstein series in terms of Deligne cohomology.

This work is built on several fundamental results proved by Zucker in the
1980s---most notably, his construction in \cite{zucker} of a mixed Hodge complex
for computing the mixed Hodge structure (MHS) on the cohomology of an affine
curve with coefficients in a polarized variation of MHS, and his paper
\cite{steenbrink-zucker} with Joseph Steenbrink in which they define admissible
variations of MHS over a curve. These foundational works were generalized from
curves to higher dimensional varieties by Saito \cite{saito:mhm} and Kaswhiwara
\cite{kashiwara}.

% Fundamental to this work is a generalization \cite{vmhs} of the classification
% of unipotent variations of MHS \cite{hain-zucker} proved by Zucker with the
% author.

By an {\em affine group}, we mean an affine group scheme over a field of
characteristic zero. The ring of functions on an affine group $G$ will be
denoted by $\cO(G)$. Such a group is {\em algebraic} if $\cO(G)$ is finitely
generated. Every affine group over the field $\F$ is a proalgebraic $\F$-group;
that is, it is an inverse limit of linear affine algebraic $\F$-groups,
\cite[p.~24]{waterhouse}.

The theory of cohomology of algebraic groups was developed by Hochschild
\cite{hochschild} and Hochschild--Serre \cite{hochschild-serre}. It generalizes
to affine groups \cite{levi}. A modern exposition can be found in the early
chapters of Jantzen's book \cite{jantzen}, which we use as our basic reference.

Denote the category of representations of an affine $\F$-group $G$ by $\Rep(G)$.
We refer to objects of $\Rep(G)$ as $G$-modules.  The most direct definition of
the cohomology of an affine group $G$ with coefficients in the $G$-module $V$ is
$$
H^\dot(G,V) := \Ext^\dot_{\Rep(G)}(\F,V),
$$
where the right hand side denotes the group of Yoneda extensions of the trivial
$G$-module $\F$ by $V$ in $\Rep(G)$.

Suppose now that $\F$ is a subfield of $\R$. A MHS on an affine $\F$-group $G$
is an ind-MHS on its coordinate ring $\cO(G)$ that is respected by its Hopf
algebra operations. A Hodge representation of the affine group $G$ with an
$\F$-MHS is an $\F$-MHS $V$ whose underlying vector space is a $G$-module with
the property that the coaction
$$
V \to V \otimes \cO(G)
$$
is a morphism of MHS. Denote the category of Hodge representations of $G$ by
$\HRep(G)$. It is an abelian category. The cohomology $H^\dot(G,V)$ groups of
$G$ with coefficients in a Hodge representation have a natural ind-MHS. (Cf.\
Proposition~\ref{prop:coho_mhs}.)

Suppose that $G$ is an affine $\F$-group with a MHS. The Deligne--Beilinson
cohomology of $G$ with coefficients in the Hodge representation $V$, defined in
Section~\ref{sec:mhs}, is isomorphic to the Yoneda ext group
$$
H_\cD^\dot(G,V) := \Ext^\dot_{\HRep(G)}(\F,V),
$$
where $\F$ denotes the Hodge structure $\F(0)$, and sits in  an exact sequence
$$
0 \to \Ext^1_\MHS(\F,H^{m-1}(G,V)) \to H^m_\cD(G,V) \to \G H^m(G,V) \to 0
$$
of vector spaces, where $\G = \Hom_\MHS(\F,\blank)$.

Affine groups with MHS and their Hodge representations arise as follows. Suppose
that $X$ is a smooth complex algebraic variety and that $\H$ is a polarized
variation of $\F$-Hodge structures over $X$. Fix a base point $x\in X$. Denote
the fiber of $\H$ over $x$ by $H_x$. The Zariski closure of the image of the
monodromy representation
$$
\rho_x : \pi_1(X,x) \to \Aut(H_x)
$$
is a reductive $\F$-group, which we denote by $R_x$.
The completion of $\pi_1(X,x)$ relative to $\rho_x$ is a proalgebraic (and thus
an affine) $\F$-group $\cG_x$. It has a natural MHS (\cite{hain:malcev}).

Denote by $\MHS(X,\H)$ the category of admissible variations $\V$ of MHS over
$X$ with the property that the monodromy action of $\pi_1(X,x)$ on each weight
graded quotient $\Gr^W_m V_x$ of its fiber over $x$ factors through an action of
$R_x$ on $\Gr^W_m V_x$. Restricting to the fiber over the base point defines an
equivalence of categories
$$
\MHS(X,\H) \to \HRep(\cG_x).
$$
This statement generalizes the main result of \cite{hain-zucker} of the author
and Zucker.\footnote{A proof of the more general statement for affine curves is
sketched in \cite{hain:modular}. A detailed proof of the general case will be
given in \cite{vmhs}.} Combined with the result above, this result establishes
that there is an isomorphism
$$
H^\dot_\cD(\cG_x,V) \cong \Ext^\dot_{\MHS(X,\H)}(\F,\V)
$$
that is compatible with all natural products.

In Section~\ref{sec:nat_homom}, we prove that for all $\V$ in $\MHS(X,\H)$,
the natural homomorphism
$$
H^\dot(\cG_x,V_x) \to H^\dot(X,\V)
$$
is a morphism of MHS, where the right-hand group has Saito's MHS,
\cite{saito:mhc}. In Section~\ref{sec:db_coho}, we lift this to a natural map
$$
\theta : H^\dot_\cD(\cG_x,V_x) \to H^\dot_\cD(X,\V)
$$
from the Deligne--Beilinson cohomology of $\cG_x$ to the Deligne--Beilinson
cohomology of $X$ and show that it is an isomorphism in degrees $0,1$ and
injective in degree 2.

Denote the category of all admissible variations of $\F$-MHS over $X$ by
$\MHS_\F(X)$. The following is a combination of Theorem~\ref{thm:big_thm} and
its corollary.

\begin{bigtheorem}
\label{thm:big}
If $X$ is a quasi-projective manifold and $\V$ is an admissible variation of
$\F$-MHS over $X$, there there is a homomorphism
$$
\Ext^\dot_{\MHS(X)}(\F,\V) \to H^\dot_\cD(X,\V)
$$
that is compatible with products. It is an isomorphism in degrees $\le 1$ and
injective in degree $2$ for all $X$, and an isomorphism in all degrees when $X$
is an affine curve.
\end{bigtheorem}

When $\V$ is a polarized variation of $\Z$-Hodge structure over $X$ of negative
weight, the corresponding group of admissible normal functions is, by
definition, the group of extensions of $\Z$ by $\V$ in the category of
admissible variations of $\Z$-MHS over $X$. These are holomorphic sections of
the corresponding bundle of intermediate jacobians $J(\V)$ over $X$ with good
asymptotic properties. Each admissible normal function determines a class in
$H^1(X,\V)$. A special case of the previous result is the following modest
generalization of a result \cite{zucker:normal} of Zucker.

\begin{bigcorollary}
The group of admissible normal function sections of $J(\V)$, after tensoring
with $\Q$, is isomorphic to $H^1_\cD(X,\V_\Q)$. Its image in $H^1(X,\V)$ is the
group $\G H^1(X,\V_\Q)$ of Hodge classes of type $(0,0)$ in
$H^1(X,\V)$.
\end{bigcorollary}

The primary motivation for this work was to interpret Francis Brown's
computation \cite{brown:mmv} of the periods of twice iterated integrals of
Eisenstein series in terms of the cup product in Deligne cohomology. This is
achieved in Sections~\ref{sec:brown} and \ref{sec:eis_projn}. While not apparent
in this paper, this interpretation uses Zucker's work on the Hodge theory of
modular curves \cite{zucker}. This topic is fully explained in
\cite{hain:modular}.

\noindent{\em Acknowledgments:} I am grateful to Francis Brown for patiently
explaining his period computations \cite{brown:mmv} and for his many helpful
comments and corrections. The author gratefully acknowledges the support of the
Friends of the Institute for Advanced Study which allowed him to spend the
2014--15 academic year at the Institute.

\section{Conventions}

Unless otherwise stated, all modules over a group $G$ are {\em left}
$G$-modules. If $G$ is an affine group and $V$ is a (left) $G$-module, then $V$
is naturally a {\em right} comodule over $\cO(G)$:
$$
\nabla_V : V \to V\otimes \cO(G).
$$

The category of representations of an affine $\F$-group will be denoted by
$\Rep(G)$. The category of representations of $G\otimes K$, where $K$ is an
extension field of $\F$, will be denoted $\Rep_K(G)$.

We will use the algebraists' convention for path multiplication, which means
that paths are composed in the ``functional order''. Suppose that $X$ is a
topological space. The composition $\gamma\mu$ of two paths $\gamma,\mu : [0,1]
\to X$ is defined when $\mu(1)=\gamma(0)$. The path $\gamma\mu$ first traverses
$\mu$ and then $\gamma$. 
We use this convention so that fundamental groups are compatible with Tannaka
theory. The fundamental groupoid $\Pi_X$ of $X$ is the category whose objects
are the points of $X$ and where $\Hom(x,y)$ is the set of homotopy classes of
paths from $x$ to $y$. With our path multiplication convention, there is an
equivalence of categories between local systems over $X$ and {\em left}
$\pi_1(X,x)$ modules, provided that $X$ is path connected.

When $\F$ is a subfield of the real numbers, we will denote the category of $\F$
mixed Hodge structures by $\MHS_\F$ and the trivial Hodge structure $\F(0)$ by
$\F$. For an $\F$-MHS $A$, set
$$
\G A := \Hom_\MHS(\F,A).
$$
For a subring $L$ of $\C$ that contains $\F$, $A_L$ will denote the $L$-module
that underlies $A$. 

Unless otherwise stated, all varieties, with the exception of the affine groups,
are defined over $\C$ and are of finite type.

\section{Cohomology of Affine Groups}

The cohomology theory of linear algebraic groups was initiated by Hochschild
\cite{hochschild} and Hochschild--Serre \cite{hochschild-serre}. This section is
a quick review of the generalization of this theory to affine groups. It is
mainly distilled from \cite[Chapts.~3,4]{jantzen}. The purpose of this section
is to present the theory in such a way that we will be able to check, in later
sections, that various basic constructions in the theory are compatible with
Hodge theory.

\subsection{Preliminaries}

Suppose that $\F$ is a field of characteristic zero, and that $G$ is an affine
group  over $\F$. Denote the coordinate ring of $G$ by $\cO=\cO(G)$. Denote the
category of $G$-modules by $\Rep(G)$. Every $G$-module $V$ is a direct limit of
its finite dimensional submodules. This is an immediate consequence of the fact
that if the image of $v\in V$ under the coaction
$$
\D_V : V \to V \otimes \cO
$$
is $\sum_j v_j \otimes f_j$, then the smallest submodule of $V$ that
contains $v$ is contained in the span of the $v_j$.

If $V$ is a $G$-module, then the composite
$$
\xymatrix{
V \ar[r]^(0.4){\D_V} & V\otimes\cO \ar[r]^{\id_V\otimes\e} &
V\otimes \F \ar[r]^\simeq & V
}
$$
is the identity, where $\e : \cO \to F$ denotes the augmentation (i.e.,
evaluation at the identity $e$). It follows that $\D_V$ is an inclusion.

Regard $\cO$ as a left $G$-module via right translation. On  $R$-rational
points, where $R$ is an $\F$-algebra, this is given by
$$
g : f \mapsto \{x\mapsto f(xg)\},\quad f\in\cO \text{ and } g, x \in G(R).
$$
The corresponding right comodule is given by the coproduct $\D:\cO \to
\cO\otimes\cO$. A left $G$-module homomorphism $\phi : V\to \cO$ is determined
by the composite $r_\phi := \e\circ \phi$
$$
\xymatrix{
V \ar[r]^\phi & \cO \ar[r]^\e & \F
}
$$
of $\phi$ with the augmentation. One recovers $\phi$ as the composite
$$
\xymatrix{
V \ar[r]^(0.4){\D_V} & V\otimes\cO \ar[r]^{r_\phi\otimes\id_\cO} &
\F\otimes\cO \ar[r]^(0.6)\simeq & \cO.
}
$$
This implies easily that $\cO$ is an injective object of
$\Rep(G)$. Indeed, given a diagram
$$
\xymatrix{
0 \ar[r] & A\ar[dr]_{\phi_A} \ar[r]^\psi & B \ar@{..>}[d]^{\phi_B}\cr
&& \cO
}
$$
of $G$-modules with exact top row, choose a lift $r_B : B \to \F$ of $r_A :=
r_{\phi_A} : A \to \F$. The lift $\phi_B$ is the composite
$(r_B \otimes \id_\cO)\circ\Delta_B : B \to \cO$.

A consequence of this computation is that if $V$ is a $G$-module, then the map
$\Delta_V : V \hookrightarrow V\otimes\cO$ is an embedding of $V$ into an
injective $G$-module. That is, $\Rep(G)$ has enough injectives. More generally,
one has:

\begin{lemma}
\label{lem:inj}
If $M$ is an $\F$-vector space, then the right $\cO$-comodule $M\otimes \cO$ is
an injective left $G$-module. For each $G$-module $A$, $G$-module maps $\phi : A
\to M\otimes \cO$ correspond bijectively with $\F$-linear maps $r : A \to
M$. Two such maps correspond if and only if the diagram
$$
\xymatrix{
A \ar[r]^\phi\ar[dr]_r & M\otimes \cO \ar[d]^{\id_M\otimes \epsilon} \cr
& M
}
$$
commutes.
\end{lemma}

Extensions of a $G$-module homomorphism $\phi_A : A\to M\otimes \cO$ from a
$G$-submodule $A$ of a $G$-module $B$ correspond to $\F$-linear extensions $r_B
: B \to M$ of $(\id_M\otimes\epsilon)\circ \phi: A \to M$.

\begin{remark}
As is standard, right $G$-modules can be converted into left $G$-modules using
inverses. This functor takes injectives to injectives. In particular, $\cO(G)$
is injective when viewed as both a right and as a left $G$-module. In
particular, it is injective with respect to the action defined by
$$
(gf)(x) := f(g^{-1}x),\quad g,x \in G(R),\ f \in \cO(G).
$$
\end{remark}

\subsection{Cohomology}

The cohomology of $G$ with coefficients in the $G$-module $V$ is defined to be
the group
$$
H^\dot(G,V) = \Ext^\dot_{\Rep(G)}(\F,V)
$$
of Yoneda extensions. The fact that every $G$-module is a direct limit of its
finitely generated submodules implies that if $A$ and $B$ are finite dimensional
$G$-modules, then the natural map
\begin{equation}
\label{eqn:rep_f}
\Ext^\dot_{\Rep^\fte(G)}(A,B) \to \Ext^\dot_{\Rep(G)}(A,B)
\end{equation}
is an isomorphism,
where $\Rep^\fte(G)$ denotes the category of finite dimensional $G$-modules. In
particular,
$$
H^\dot(G,V) \cong \Ext^\dot_{\Rep^\fte(G)}(\F,V)
$$
whenever $V$ is a finite dimensional $G$-module.

Since $\Rep(G)$ has enough injectives,\footnote{Note that $\Rep^\fte(G)$ does
not have enough injectives when $G$ is non-trivial.} $H^\dot(G,V)$ can be
defined by taking an injective resolution $V \to I^\dot$ of $V$ and then taking
the homology of its invariants:
$$
H^\dot(G,V) = H^\dot(\Hom_G(\F,I^\dot)) = H^\dot\big((I^\dot)^G\big).
$$

\begin{remark}
Note that an algebraic group $G$ is reductive if and only if $H^j(G,V)$ vanishes
for all $G$-modules $V$ and all $j>0$.
\end{remark}

\subsection{Standard cochains}

Denote by $E_\dot G$ the simplicial scheme whose scheme of $n$-simplices $E_n G$
is $G^{n+1}$. The group $G$ acts on each $E_n G$ diagonally on the right. The
face maps
$$
d_j : E_n G \to E_{n-1}G,\ j\in\{0,\dots,n\}, \ n>0
$$
are defined by $d_j (x_0,\dots,x_n) = (x_0,\dots,\widehat{x}_j,\dots,x_n)$.
They are $G$-equivariant.

Dually, one has the cosimplicial object $\cO^{\otimes \dot}$ of $\Rep(G)$
whose module in degree $n$ is
$$
\cO(E_n G) = \cO(G^{n+1}) \cong \cO^{\otimes(n+1)}.
$$
The map $G\times G^n \to G^{n+1}=E_nG$ defined by
$$
\big(g,(g_1,\dots,g_n)\big) \mapsto (g,gg_1,gg_1g_2,\dots,gg_1\cdots g_n)
$$
is $G$-equivariant, where $G$ acts on $G\times G^n$ by left multiplication on
the first factor and trivially on $G^n$. The discussion above implies that each
$\cO(E_n G) \cong \cO(G^n)\otimes \cO(G)$ is an injective left $G$-module with
respect to the action\footnote{This argument shows that $\cO^{\otimes n}$ is an
injective $G$ module for all $n\ge 1$, where $G$ acts diagonally. It implies
that the tensor product $\cO^{\otimes n}\otimes \cO^{\otimes m}$ of two such
$G$-modules is also injective.}
$$
(g f)(g_0,\dots,g_n) = f(g^{-1}g_0,\dots,g^{-1}g_n),\quad f \in \cO(G^{n+1}).
$$

Denote the coface maps of $\cO^{\otimes\dot}$ by $d^j : \cO^{\otimes n} \to
\cO^{\otimes(n+1)}$. Define a differential
$$
d : \cO^{\otimes n} \to \cO^{\otimes(n+1)}
$$
by $d = \sum_{0\le j \le n} (-1)^j d^j$. The section $\sigma : E_n G \to
E_{n+1}G$ defined by $\sigma(x_0,\dots,x_n) = (e,x_0,\dots,x_n)$ induces the
$G$-module map $s : \cO^{\otimes(n+1)} \to \cO^{\otimes n}$ defined by
$$
s(f_0\otimes f_1 \otimes \dots \otimes f_n) \mapsto \epsilon(f_0)\, f_1 \otimes
\dots \otimes f_n.
$$
It satisfies $ds + sd = \id - \epsilon$. Consequently, the inclusion $V \to
V\otimes \cO(E_0 G) = V\otimes \cO$ of the constant $V$-valued functions gives
an injective resolution
$
0 \to V \to V\otimes \cO^{\otimes \dot}
$
of $V$ as a left $G$-module.

Thus $H^\dot(G,V)$ can be computed as the cohomology of the complex
\begin{equation}
\label{eqn:std_cochains}
\cC^\dot(G,V) := \Hom_G(\F,V\otimes\cO^{\otimes\dot}) = \Mor_G(E_\dot G,V),
\end{equation}
where $G$ acts on the left of $E_\dot G$ via $g : (g_0,\dots,g_n) \mapsto
(gg_0,\dots,gg_n)$. Identifying $E_n G$ with $G\times G^n$ as above, we see 
that $H^\dot(G,V)$ can be computed using the standard bar complex whose cochains
in degree $n$ are
$$
\cC^n(G,V) = \Mor(G^n,V)
$$
with the usual differential $\delta : \Mor(G^n,V) \to \Mor(G^{n+1},V)$.
Cf.\ Maclane \cite[Ch.~IV, \S4]{maclane} or Brown \cite[Ch.~I,\S5]{brown:book}.

\begin{remark}
\label{rem:qism}
The complex $\cC^\dot(G,V)$ is functorial in $G$ and $V$, and exact in $V$.
If $V^\dot$ is a complex in $\Rep(G)$ that is bounded below, one can
define $H^\dot(G,V^\dot)$ to be the hyper cohomology of the total complex
associated to the double complex
$$
\cC^\dot(G,V^\dot).
$$
Standard homological arguments imply that if $V_1^\dot \to V_2^\dot$ is a 
quasi-isomorphism, then $H^\dot(G,V_1^\dot) \to H^\dot(G,V_2^\dot)$ is an
isomorphism.
\end{remark}

Cohomology behaves well with respect to limits. Suppose that the affine group
$G$ is the inverse limit $\varprojlim_\alpha G_\alpha$ of algebraic groups
$G_\alpha$. Since
$$
\cO(G) = \varinjlim_\alpha \cO(G_\alpha)
$$
and since every $G$-module $V$ can be written as a direct limit of
$G_\alpha$-modules $V_\alpha$,
$$
\big(V\otimes\cO(G)^{\otimes\dot}\big)^G =
\varinjlim_\alpha \big(V\otimes\cO(G_\alpha)^{\otimes\dot}\big)^{G_\alpha}.
$$
Since homology commutes with direct limits:

\begin{proposition}
If the affine group $G$ is an inverse limit of algebraic groups $G_\alpha$ as
above, then for all $G$-modules $V$, we can write $V = \varinjlim_\alpha
V_\alpha$ where $V_\alpha$ is a $G_\alpha$-module and there is a natural
isomorphism
$$
\varinjlim H^\dot(G_\alpha,V_\alpha) \overset{\simeq}{\longrightarrow}
H^\dot(G,V).
$$
\end{proposition}

\subsection{The spectral sequence of a group extension}
\label{sec:ss}

Suppose that $N$ is a normal subgroup of the affine group $G$. The conjugation
action of $G$ on $N$ induces an action of $G$ on $H^\dot(N,V)$. The usual
argument shows that its restriction to $N$ is trivial, so that $H^\dot(N,V)$ is
a $G/N$-module.

Taking a $G$-module $V$ to its $N$-invariants $V^N$ defines a functor $Q:\Rep(G)
\to \Rep(G/N)$. It takes injectives to injectives. Denote the category of
$\F$-vector spaces by $\Vec_\F$. Applying the Grothendieck spectral sequence to
the diagram of functors
$$
\xymatrix{
\Rep(G) \ar[r]_(.45)Q \ar@/^1.5pc/[rrr]^{\Hom_G(\F,\blank)} &
\Rep(G/N) \ar[rr]_(.55){\Hom_{G/N}(\F,\blank)} && \Vec_\F
}
$$
gives the spectral sequence for a group extension:

\begin{proposition}
\label{prop:hs-ss}
If $N$ is a normal subgroup of the affine group $G$, then for all
$G$-modules $V$, there is a spectral sequence satisfying
$$
E^{s,t}_2 = H^s(G/N,H^t(N,V)) \Longrightarrow H^{s+t}(G,V).
$$
\end{proposition}

\begin{corollary}
\label{cor:coho}
If the affine group $G$ is an extension
$$
1 \to N \to G \to R \to 1
$$
of a reductive group by an affine group $N$, then there is a natural
isomorphism
$$
H^\dot(G,V) \cong H^\dot(N,V)^R.
$$
\end{corollary}

\section{Mixed Hodge Structures on Affine Group Schemes}
\label{sec:mhs}

Throughout this section, $\F$ will be a subfield of $\R$. Denote the category of
$\F$-mixed Hodge structures by $\MHS_\F$. It is a neutral tannakian category.
Let $\w^B : \MHS_\F \to \Vec_\F$ be the functor that takes a mixed Hodge
structure to its underlying $\F$-vector space. It is a fiber functor. (The
``Betti realization''.) Set
$$
\pi_1(\MHS_\F) := \Aut^\otimes \w^B.
$$
By Tannaka duality \cite[Thm.~2.11]{deligne-milne}, $\Rep(\pi_1(\MHS_\F))$ is
equivalent to the category of ind-objects of $\MHS_\F$. The category $\MHS_\F$
is equivalent to $\Rep^\fte(\pi_1(\MHS_\F))$ of finite dimensional
$\pi_1(\MHS_\F)$-modules.

An $\F$-mixed Hodge structure on an affine $\F$-group $G$ is, by definition, an
action of $\pi_1(\MHS_\F)$ on $G$. Equivalently, the coordinate ring $\cO(G)$
of $G$ is a Hopf algebra in $\Rep(\pi_1(\MHS_\F))$, the category of ind-objects
of $\MHS_\F$.

A {\em Hodge representation} of a group $G$ with a mixed Hodge structure is an
ind-object $V$ of $\MHS_\F$ together with an action of $G$ on $V$ such that the
coaction
$$
\D_V : V \to V\otimes\cO(G)
$$
is a morphism in ind-$\MHS_\F$. Denote the category of Hodge representations of
$G$ by $\HRep(G)$.

If $G$ has a mixed Hodge structure, then one can form the semi-direct product
$$
\Ghat := \pi_1(\MHS_\F) \ltimes G,
$$
which is also an affine $\F$-group. Note that $\Rep(\Ghat) = \HRep(G)$.

\begin{proposition}
\label{prop:coho_mhs}
If $G$ is an affine group with a mixed Hodge structure, and if $V$ is a Hodge
representation of $G$, then the cohomology group $H^\dot(G,V)$ has a natural
ind-MHS. It is compatible with products in the sense that if $V_1$ and $V_2$ are
Hodge representations of $G$, then the product
\begin{equation}
\label{eqn:prod}
H^\dot(G,V_1) \otimes H^\dot(G,V_2) \to H^\dot(G,V_1\otimes V_2).
\end{equation}
is a morphism of MHS.
\end{proposition}

\begin{proof}
Standard homological arguments (see Section~\ref{sec:ss}) imply that
$H^\dot(G,V)$ is a $\Ghat/G = \pi_1(\MHS_\F)$-module. So $H^\dot(G,V)$ is an
ind-object of $\MHS_\F$. They also imply that the product (\ref{eqn:prod}) is a
$\pi_1(\MHS_\F)$-module homomorphism.
\end{proof}

In particular, if $V$ is in $\HRep(G)$, then $V^G$ is a sub-MHS of $V$.

\begin{proposition}
If $G$ is an affine group with an $\F$-mixed Hodge structure and if $V$ is a
Hodge representation of $G$, then there is a natural exact sequence
\begin{equation}
\label{eqn:ses_ext}
0 \to \Ext^1_{\MHS_\F}(\F,H^{j-1}(G,V)) \to \Ext^j_{\HRep(G)}(\F,V)
\to \G H^j(G,V) \to 0.
\end{equation}
\end{proposition}

\begin{proof}
Since $\Ext^j_{\MHS_\F}(A,B)$ vanishes whenever $j>1$, the result follows by
applying Proposition~\ref{prop:hs-ss} to the extension
$1 \to G \to \Ghat \to \pi_1(\MHS_\F) \to 1$.
\end{proof}

\subsection{Cochains}

Here we dispose of several technical issues that arise in subsequent sections.
Throughout this section, $G$ is an affine group with an $\F$-MHS.

The first task is to give an alternative construction of the MHS on
$H^\dot(G,V)$ for each Hodge representation $V$ of $G$. Note that $\HRep(G)$ has
enough injectives as $\HRep(G) = \Rep(\Ghat)$.

\begin{definition}[cf.\ \cite{carlson-hain}]
\label{def:rel_inj}
A {\em relatively injective $G$-module} is an object $I$ of $\HRep(G)$ that is
an injective $G$-module. A {\em relatively injective resolution} of an object
$V$ of $\HRep(G)$ is a resolution $V \to I^\dot$ in $\HRep(G)$ which is an
injective resolution of $V$ in $\Rep(G)$. Equivalently, it is a resolution in
$\HRep(G)$ in which each $I^n$ is a relatively injective $G$-module.
\end{definition}

Injective modules of the form $V\otimes \cO(G)^{\otimes n}$, where $V$ is in
$\HRep(G)$, are relatively injective. The standard resolution $V \to V\otimes
\cO(G)^{\otimes \dot}$ of $V$ in $\Rep(G)$ is a relatively injective resolution
of $V$ in $\HRep(G)$.

\begin{proposition}
\label{prop:qism}
If $V \to I^\dot$ is a relatively injective resolution of $V$ in $\HRep(G)$,
then $(I^\dot)^G$ is a complex in ind-$\MHS_\F$ and there is a natural
isomorphism
$$
H^\dot(G,V) \cong H^\dot((I^\dot)^G)
$$
in ind-$\MHS_F$.
\end{proposition}

\begin{proof} Let $J^{s,t} = I^t\otimes \cO(\Ghat)^{\otimes(s+1)}$. Then $0 \to
I^t \to J^{\dot,t}$ is an injective resolution of $I^t$ in $\Rep(\Ghat)$, $V\to
\tot J^{\dot\dot}$ is an injective resolution of $V$ in $\Rep(\Ghat)$, and 
$I^\dot \to \tot J^{\dot\dot}$ is a quasi-isomorphism in $\Rep(\Ghat)$.

Since $\cO(\Ghat) \cong \cO(G)\otimes \cO(\pi_1(\MHS_\F))$ as a $G$-module,
$\cO(\Ghat)$ is an injective $G$-module. Consequently, $V \to \tot J^{\dot\dot}$
is an injective resolution in $\Rep(G)$, which implies that
$$
(I^\dot)^G \to \tot (J^{\dot\dot})^G
$$
is a quasi-isomorphism. It induces an isomorphism of MHS on cohomology. The
right hand complex is a complex of $\Ghat/G \cong \pi_1(\MHS)$-modules which
computes the natural MHS on $H^\dot(G,V)$. The result follows.
\end{proof}

The exactness properties of $\Gr^W_\dot$ and $\Gr_F^\dot$ imply that relatively
injective modules have the following property:

\begin{lemma}
\label{lem:lifts}
Suppose that $0 \to A \to B$ is an exact sequence in $\HRep(G)$. If $I$ is
a relatively injective $G$-module, then the natural surjection
$$
\Hom_G(B,I) \to \Hom_G(A,I)
$$
is strict with respect to the Hodge and weight filtrations $F^\dot$ and
$W_\dot$. In particular the maps
\begin{multline*}
W_0\Hom_G(B_\F,I_\F) \to W_0\Hom_G(A_\F,I_\F),\
W_0\Hom_G(B_\C,I_\C) \to W_0\Hom_G(A_\C,I_\C)
\cr \text{ and }
F^0W_0\Hom_G(B_\C,I_\C) \to F^0W_0\Hom_G(A_\C,I_\C)
\end{multline*}
are surjective. \qed
\end{lemma}

%% The following fact will be useful when studying products in Deligne
%% cohomology.

%%\begin{lemma}
%%The tensor product of standard resolutions is injective. So tensor product of
%%standard resolutions are relatively injective.
%%\end{lemma}

\subsection{Deligne--Beilinson cohomology of an affine group}

It is convenient to have a complex which computes $\Ext^\dot_{\HRep(G)}(\F,V)$
which depends only on $G$ and its MHS. This is the complex of Deligne cochains
defined below.

\begin{definition}
The {\em Deligne--Beilinson cohomology} $H_\cD^\dot(G,V)$ of an affine group
$G$ with an $\F$-mixed Hodge structure with coefficients in a Hodge
representation $V$ is defined to be the cohomology of the complex
\begin{equation}
\label{eqn:dbcoho-group}
\cone\big(F^0W_0\cC^\dot(G,V)_\C \oplus W_0 \cC^\dot(G,V)_\F
\to W_0\cC^\dot(G,V)_\C\big)[-1].
\end{equation}
\end{definition}

As in the case of classical Deligne-Beilinson cohomology, it sits in an
exact sequence
\begin{equation}
\label{eqn:ses_db}
0 \to \Ext^1_{\MHS_\F}(\F,H^{j-1}(G,V)) \to H^j_\cD(G,V) \to \G H^j(G,V) \to 0.
\end{equation}

Proposition~\ref{prop:qism} implies that the DB-cohomology of $G$ can be
computed with any relatively injective resolution of $V$.

The following result generalizes the main result of \cite{carlson-hain}.

\begin{theorem}
\label{thm:del_is_ext}
If $G$ is an affine group with an $\F$-mixed Hodge structure, and if $V$ is a
Hodge representation of $G$, then there is a natural isomorphism
$$
\Ext^\dot_{\HRep_\F(G)}(\F,V) \overset{\simeq}{\longrightarrow} H^\dot_\cD(G,V).
$$
\end{theorem}

\begin{proof}[Sketch of Proof]
This is a refinement of the standard proof that the Yoneda product in Ext groups
corresponds to the cup product in cohomology constructed using injective
resolutions. To prove the result, it suffices to construct a homomorphism
\begin{equation}
\label{eqn:nat_isom}
\Ext^\dot_{\HRep_\F(G)}(\F,V) \to H^\dot_\cD(G,V)
\end{equation}
that induces a map from the exact sequence (\ref{eqn:ses_ext}) to the exact
sequence (\ref{eqn:ses_db}) and is an isomorphism on the kernels and cokernels.

Let
$$
0 \to V \to E^0 \to E^1 \to \cdots \to E^{n-1} \to F \to 0
$$
be a Yoneda $n$-extension in $\HRep_\F(G)$. Then one has the standard resolution
$$
0 \to V \to I^0 \to I^1 \to I^2 \to \cdots
$$
of $V$ in $\Rep(G)$, where $I^n = V\otimes\cO(G)^{\otimes(n+1)}$. This is a
relatively injective resolution of $V$ in $\HRep(G)$.

\def\ardd#1{{\ar[d]_{\phi_{#1}^\DR}\ar@<0.75ex>[d]^{\phi_{#1}^\betti}}}

Standard homological arguments and Lemma~\ref{lem:lifts} imply that we have the
diagram:
$$
\xymatrix@C=30pt@R=32pt{
0 \ar[r] & V \ar[r]\ar@{=}[d] & E^0 \ar[r]\ardd{0}\ar[dl]_{s_0} &
E^1 \ardd{1}\ar[dl]_(0.4){s_1} & \cdots & 
E^{n-1} \ar[r]\ardd{n-1} &
F \ar[r]\ardd{n} \ar[dl]_(0.4){s_n} & 0 \ar[d]
\cr 
0 \ar[r] & V \ar[r] & I^0 \ar[r] & I^1 
& \cdots & 
I^{n-1} \ar[r] & I^n \ar[r] & I^{n+1}
}
$$
where $\phi_\dot^\DR$ and $\phi_\dot^\betti$ are chain maps satisfying
$$
\phi_j^\DR \in F^0W_0 \Hom_\C(E^j,I^j)^G
\text{ and } \phi_j^\betti \in W_0\Hom_\F(E^j,I^j)^G
$$
(convention $E^n = F$) and $s_\dot$ is a chain homotopy from $\phi_\dot^\DR$
to $\phi_\dot^\betti$ satisfying
$$
s_j \in W_0\Hom_\C(E^j,I^{j-1})^G,\quad s_0 = 0.
$$
Since $\delta (s_n) = \phi_n^\DR - \phi_n^\betti$, the triple
$(\phi_n^\DR,\phi_n^\betti,s_{n})$ is a cocycle in the complex
(\ref{eqn:dbcoho-group}). The homotopy properties of injective resolutions imply
that the cohomology class of this cocycle is independent of the choices of
$\phi_\dot^\DR$, $\phi_\dot^\betti$ and $s_\dot$.\footnote{A detailed proof of
the dual assertion in the unipotent case is given in
\cite[\S8]{carlson-hain}.} There is therefore a well defined map
$$
\Ext^\dot_{\HRep(G)}(\F,V) \to H^\dot_\cD(G,V).
$$
It is easily seen to be a homomorphism with the property that it induces the
identity on the kernels and cokernels of the exact sequences (\ref{eqn:ses_ext})
and (\ref{eqn:ses_db})
\end{proof}

\subsection{Cochains}

Suppose that $A^\dot$ is a complex in ind-$\MHS_\F$. Then one can consider the
complex
$$
A^\dot_\cD := \cone\big(F^0W_0 A^\dot_\C 
\oplus W_0 \A^\dot_\F \to W_0\A^\dot_\C\big)[-1].
$$
Following Beilinson \cite{beilinson}, we write an element of degree $j$ in the
form
$$
\begin{bmatrix}
& c & \cr w & & z
\end{bmatrix}
$$
where $w\in F^0W_0 A^j_\C$, $z\in W_0 A^j_\F$ and $c\in W_0 A^{j-1}_\F$. With
this notation, the differential is
$$
\delta :
\begin{bmatrix}
& c & \cr w & & z
\end{bmatrix}
\mapsto
\begin{bmatrix}
& -dc + w - z & \cr dw & & dz
\end{bmatrix}.
$$

\subsection{Products}

Suppose that we have complexes $A^\dot$, $B^\dot$ and $C^\dot$ in ind-$\MHS_\F$.
Suppose that there is an associative product
$$
\cup : A^\dot \otimes B^\dot \to C^\dot
$$ 
which is a morphism in ind-$\MHS_\F$ and commutes with the differentials. For
each $t \in \F^\times$, define a product
$$
\cup_t : A_\cD^\dot \otimes B^\dot_\cD \to C^\dot_\cD
$$
by
$$
\begin{bmatrix}
& c' & \cr w' & & z'
\end{bmatrix}
\otimes
\begin{bmatrix}
& c'' & \cr w'' & & z''
\end{bmatrix}
\mapsto
\begin{bmatrix}
& c_t & \cr w'\cup w'' & & z' \cup z''
\end{bmatrix}
$$
where
$$
c_t = (1-t)c_0 + t c_1
$$
and
$$
c_0 = c'\cup z'' + (-1)^{|w'|} w'\cup c''
\text{ and }
c_1 = c'\cup w'' + (-1)^{|z'|} z'\cup c''.
$$
Each $\cup_t$ commutes with the differentials and is chain homotopic to $c_0$,
so that all $\cup_t$ induce the same map on cohomology.

\begin{proposition}
The natural isomorphism (\ref{eqn:nat_isom}) is compatible with products.
\end{proposition}

\begin{proof}
Suppose that
\begin{equation}
\tag{$\be_j$}
0 \to V_j \overset{\eta_j}{\longrightarrow} E_j^0 \longrightarrow \dots
\longrightarrow E_j^{n_j-1} \longrightarrow F \longrightarrow 0,
\quad j = 1,2
\end{equation}
are two Yoneda extensions in $\HRep(G)$ with $n_j>0$. Let
$$
0 \to V_j \to I_j^\dot, \qquad j = 1,2
$$
be the standard resolution of $V_j$ in $\Rep(G)$. These are relatively injective
resolutions. Their tensor product
$$
0 \to V_1 \otimes V_2 \to \tot (I_1^\dot \otimes I_2^\dot)
$$
is a resolution of $V_1\otimes V_2$ in $\HRep(G)$. Since $\cO^{\otimes k}$ is
an injective object of $\Rep(G)$ for all $k > 0$, each
$$
I_1^j\otimes I_2^k \cong V_1\otimes V_2 \otimes \cO^{\otimes(j+k+2)}
$$
is injective in $\Rep(G)$. It follows that
$
I^\dot := \tot (I_1^\dot \otimes I_2^\dot)
$
is a relatively injective resolution of $V_1\otimes V_2$ in $\HRep(G)$ and that
it can be used to compute $H^\dot_\cD(G,V_1\otimes V_2)$.

For convenience we set $m=n_1$ and $n=n_2$. Let
$$
(\phi^\DR_j,\ \phi^\betti_j,\sigma_j),\qquad j=0,\dots,m
$$
be two chain maps and a homotopy between them associated to the extension
($\be_1$) as in the proof of Theorem~\ref{thm:del_is_ext}. Let 
$$
(\psi^\DR_k,\ \psi^\betti_k,\ \tau_k),\qquad k=0,\dots,n
$$
the corresponding chain maps and homotopy associated to ($\be_2$). Note that
$\sigma_0 = \tau_0 = 0$. The extensions ($\be_1$) and ($\be_2$) are represented
by the Deligne cocycles
\begin{equation}
\label{eqn:cocycles}
\begin{bmatrix}
& \sigma_m \cr \phi^\DR_m && \phi^\betti_m
\end{bmatrix}
\text{ and }
\begin{bmatrix}
& \tau_n \cr \psi^\DR_n && \psi^\betti_n
\end{bmatrix}
\end{equation}
respectively.

\def\aruu{{\ar@{.>}[u]\ar@<0.75ex>@{.>}[u]}}  %% replaces \ar@{.>}[u]
\def\arll{{\ar@{.>}[l]\ar@<0.75ex>@{.>}[l]}}  %% replaces \ar@{.>}[l]

\begin{figure}[ht!]
{\tiny
$$
\xymatrix@C=20pt@R=20pt{
&&&&& &0
\cr
&& &\ar[r] & I_1^m\otimes I_2^{n-1} \ar[u] \ar[r] & I_1^m\otimes I_2^n \ar[u]
 & \F \otimes \F \arll \ar[u] \ar@{.>}[dl] \ar@/_1pc/@{.>}[ll]
\cr
&&& \ar[r] & I_1^{m-1}\otimes I_2^{n-1} \ar[u] \ar[r] &
I^{m-1}_1\otimes I_2^n  \ar[u]  &
E_1^{m-1}\otimes \F \arll\ar[u] \ar@/_1pc/@{.>}[ll]
\cr
&& \vdots & \vdots & \vdots & \vdots & \vdots
\cr
&&&& I_1^1\otimes I_2^{n-1}  \ar[r] & I_1^1\otimes I_2^n  &
E_1^1\otimes \F \arll \ar@{.>}[dl]^(.4)0 \ar@/_1pc/@{.>}[ll]
\cr
& V_1\otimes V_2 \ar[r]^{\eta_1\otimes\eta_2} &
I_1^0\otimes I_2^0 \ar[r]\ar[u] &
I_1^0\otimes I_2^1\ar[u] & 
I_1^0\otimes I_2^{n-1} \ar[u] \ar[r] & I_1^0\otimes I_2^n \ar[u] &
E_1^0\otimes \F \arll \ar[u] \ar@/_1pc/@{.>}[ll]
\cr
0 \ar[r] &V_1\otimes V_2 \ar[r]\ar@{=}[u] &
V_1\otimes E_2^0 \ar[r] \aruu \ar@{.>}[ul]^0 &
V_1\otimes E_2^1 \aruu \ar@{.>}[ul] &
V_1 \otimes E_2^{n-1} \ar@{->>}[r]\aruu &
V_1 \otimes \F \aruu\ar@{^{(}->}[ur] \ar@{.>}[ul] %%% ) for balance
}
$$
}
\caption{}
\label{fig:diag}
\end{figure}

We need to show that the Deligne cohomology class that corresponds to their
Yoneda product is a cup product of these cocycles. To this end, consider the
diagram in Figure~\ref{fig:diag}. The bottom and right hand edges form the
Yoneda product of ($\be_1$) and $(\be_2)$. We have to compute the corresponding
chain maps $\Phi^\DR_\dot$, $\Phi_\dot^\betti$ and the homotopy $s_\dot$ for
their Yoneda product.

For $?\in\{\DR,B\}$, set
$$
\Phi^?_j =
\begin{cases}
\eta_1\otimes \psi_j^? & j=0,\dots,n-1,\cr
\phi_k^?\otimes \psi_n^? & j=n+k,\ k = 0,\dots,m.
\end{cases}
$$
Both are $G$-invariant chain maps to $\tot (I_1^\dot\otimes I_2^\dot)$. Clearly
$\Phi^\DR_\dot$ preserves the Hodge and weight filtrations, while
$\Phi^\betti_\dot$ is defined over $\F$ and preserves the weight filtration. A
chain homotopy between them is defined by
$$
s_j =
\begin{cases}
\eta_1 \otimes \tau_j & j = 0,\dots,n-1,\cr
\sigma_k \otimes \psi^\betti_n + (-1)^k \phi_k^\DR\otimes \tau_n &
j = n+k,\ k = 0,\dots,m.
\end{cases}
$$
Note that the restriction of $s_n : E_1^0 \otimes \F \to I^1_0\otimes I_2^{n-1}$
to $V_1\otimes \F$ is $\eta_1\otimes \tau_n$ as the restriction of $\phi_0^\DR$
to $V_1$ is $\eta_1$ since $\sigma_0 = 0$.

The Yoneda product of ($\be_1$) and ($\be_2$) is thus represented by the
Deligne cocycle
$$
\begin{bmatrix}
& \sigma_m\otimes \psi^\betti_n + (-1)^m\phi^\DR_m\otimes \tau_n \cr
\phi^\DR_m\otimes \psi^\DR_n && \phi^B_m \otimes \psi^B_n
\end{bmatrix}.
$$
This is the $\cup_0$ product of the cocycles (\ref{eqn:cocycles}) that
represent ($\be_1$) and ($\be_2$).
\end{proof}

\subsection{Natural hypotheses}
\label{sec:nat-hypoth}

The MHSs on the affine groups $G$ that arise in practice (e.g., in the next
section) satisfy $W_{-1}\cO(G) = 0$. In this case, $W_0\cO(G)$ is a Hopf
subalgebra. Set $\Gbar = \Spec W_0 \cO(G)$. This is a quotient of $G$.

\begin{lemma}
The kernel of the quotient homomorphism $G \to \Gbar$ is prounipotent. Denote it
by $W_{-1}G$. It has a natural weight filtration
$$
W_{-1}G \supseteq W_{-2}G \supseteq W_{-3}G \supseteq \cdots
$$
which is a descending central series.
\end{lemma}

\begin{proof}
Denote the maximal ideal in $\cO(G)$ of functions that vanish at the identity by
$\m_G$. The maximal ideal of the identity of $\Gbar$ in $\cO(\Gbar) = W_0\cO(G)$
is $W_0\m_G$. The coordinate ring of the kernel $W_{-1}G$ of $G \to \Gbar$ is
thus $\cO(G)/(W_0 \m_G)$.

Since $\cO(G)$ is a Hopf algebra in the category of ind-objects of $\MHS_\F$,
$\cO(W_{-1}G)$ inherits a natural ind-MHS. In particular, it has a weight
filtration $W_\dot$ which is preserved by the product and coproduct. Using the
unit and counit, we can write it as
$$
\cO(W_{-1}G) = F \oplus W_1\cO(W_{-1}G),
$$
so that it is connected. Since a finitely generated connected Hopf algebra is
the ring of functions on a unipotent group, and since $\cO(W_{-1}G)$ is a direct
limit of finitely generated (necessarily connected) Hopf algebras, $W_{-1}G$ is
prounipotent.
\end{proof}

\section{Categories of Variations of MHS}
\label{sec:vmhs}

\subsection{Relative completion}
\label{sec:rel_comp}

Here we recall Deligne's notion of the relative unipotent completion of a
discrete group. Suppose that $\G$ is a discrete group, that $\F$ is a field of
characteristic zero, that $R$ is a reductive $\F$-group and that $\rho : \G \to
R(\F)$ is a Zariski dense representation. The completion of $G$ relative to
$\rho$ is an affine $\F$-group $\cG$ that is an
extension
$$
1 \to \U \to \cG \to R \to 1
$$
of $R$ by a prounipotent group. It can be defined as the tannakian fundamental
group of the category $\cR(\G,\rho)$ of finite dimensional $F[\G]$-modules $V$
that admit a filtration
$$
0 = V_0 \subset V_1 \subset \cdots \subset V_{N-1} \subset V_N = V
$$
by $\G$-submodules such that the action of $\G$ on each $V_m/V_{m-1}$ factors
through a representation of $R$ via $\rho$. There is a canonical representation
$\G \to \cG(\F)$ whose composition with the quotient mapping $\cG(\F) \to R(\F)$
is $\rho$.

\subsection{Hodge theory}

In the remainder of this section, $\F$ will be a subfield of $\R$. Suppose that
$X$ is a smooth quasi-projective variety over $\C$ and that $\H$ is a
semi-simple variation of $\F$-MHS over $X$. That is, it is a direct sum of
polarized variations of $\F$-Hodge structure (PVHS) whose local monodromy
operators are quasi-unipotent.

Fix a base point $x\in X$. Denote the fiber of the variation of mixed Hodge
structure $\V$ over $X$ over $x$ by $V_x$. Let $R_x$ be the closure of the
monodromy representation $\pi_1(X,x) \to \Aut(H_x)$. A theorem of Deligne
\cite[4.2.6]{deligne:hodge2} (see also \cite[7.2.5]{schmid}) implies that $R_x$
is a reductive $\F$-group. Let $\rho_x : \pi_1(X,x) \to R_x(\F)$ be the
corresponding representation. Denote the completion of $\pi_1(X,x)$ with respect
to $\rho_x$ by $\cG_x$. This group has a natural $\F$-MHS, which was constructed
in \cite{hain:malcev}. A more concrete description of it in the case where $X$
is an affine curve is given in \cite{hain:modular}.\footnote{A more concrete
description in the general case will be given in \cite{vmhs}.} It is an
extension
$$
1 \to \U_x \to \cG_x \to R_x \to 1
$$
where $\U_x = W_{-1}\cG_x$ is prounipotent.

Denote by $\MHS(X,\H)$ the category of admissible variations $\V$ of MHS over
$X$ whose weight graded quotients have the property that their monodromy factors
through a representation of $R_x$ via $\rho_x$. The fiber $V_x$ of such a local
system $\V$ is an object of the category $\cR(\pi_1(X,x),\rho_x)$ and thus a
representation of $\cG_x$. Its monodromy representation factors through this
homomorphism. That is, there is a homomorphism $\cG_x \to \Aut V_x$ such that
$$
\xymatrix{
\pi_1(X,x) \ar[r] & \cG_x(\F) \ar[r] & \Aut(V_x)
}
$$
is the monodromy representation.

\subsection{Relative completion of path torsors}

Suppose that $X$ is a locally simply connected topological space, that $R_x$ is
a reductive $\F$-group and that $\rho : \pi_1(X,x) \to R_x(\F)$ is a Zariski
dense representation. Denote by $\cR(X,\rho)$ the category of $\F$-local systems
$\V$ over $X$ whose monodromy representation $\pi_1(X,x) \to \Aut(V_x)$ is an
object of the category $\cR(\pi_1(X,x),\rho)$. This category is equivalent to
$\cR(\pi_1(X,x),\rho)$. For each $y\in X$ the functor
$$
\w_y : \cR(X,\rho) \to \Vec_\F
$$
that takes a local system $\V$ to the $\Q$-vector space that underlies its fiber
$V_y$ over $y$ is a fiber functor.

Denote by $\Pi_X(x,y)$ the set of homotopy classes of paths in $X$ from $x$ to
$y$. The completion of $\Pi_X(x,y)$ relative to $\rho$ is defined to be the
affine scheme
$$
\cG_{x,y} := \Isom^\otimes (\w_x,\w_y)
$$
of tensor isomorphisms of $\w_x$ and $\w_y$. When $x=y$, this is the completion
$\cG_x$ of $\pi_1(X,x)$ relative to $\rho$. There are natural functions
$\Pi_X(x,y) \to \cG_{x,y}(\F)$, which are Zariski dense, and morphisms
$\cG_{y,z} \times \cG_{x,y} \to \cG_{x,z}$ such that the diagram
$$
\xymatrix{
\Pi_X(y,z) \times \Pi_X(x,y) \ar[r]\ar[d] & \Pi_X(x,z)\ar[d] \cr
\cG_{y,z}(\F) \times \cG_{x,y}(\F) \ar[r] & \cG_{x,z}(\F)
}
$$
commutes.

\subsection{Categories of variations of MHS}

The following theorem generalizes the main result of \cite{hain-zucker}, which
is the unipotent case. This result will be proved in full generality in
\cite{vmhs}. The proof of the case where $X$ is an affine curve case was
sketched in \cite{hain:modular}.

\begin{theorem}
\label{thm:vmhs}
Suppose that $X$ is a smooth quasi-projective variety and that $x\in X$. If $\V$
is an object of $\MHS(X,\H)$, then the morphism
$$
V_x \to V_x \otimes \cO(\cG_x)
$$
corresponding to the homomorphism $\cG_x \to \Aut V_x$ is a morphism of MHS. The
corresponding functor
$$
\MHS(X,\H)\to \HRep(\cG_x)
$$
is an equivalence of categories.
\end{theorem}

The relevance of this result is that, combined with
Theorem~\ref{thm:del_is_ext},  it gives a computation of 
$\Ext^\dot_{\MHS(X,\H)}(\F,\V)$.

\begin{corollary}
\label{cor:vmhs}
If $\V$ is an object of $\MHS(X,\H)$, there is a natural isomorphism
$$
H^\dot_\cD(\cG_x,V_x) \cong \Ext^\dot_{\MHS(X,\H)}(\F,\V)
$$
which is compatible with products.
\end{corollary}

\subsection{Tangential base points}

Theorem~\ref{thm:vmhs} also holds when $x$ is a tangential base point. Because
the weight filtration is replaced by a relative weight filtration, which depends
on the tangent vector, additional explanation is required.

Suppose that $X = \Xbar - D$, where $\Xbar$ is smooth and projective and $D$ is
a normal crossing divisor in $\Xbar$. Suppose that $P\in D$. A tangent vector
$\vv \in T_P\Xbar$ at $P$ that is not tangent to any local component of $D$
determines an element $\sigma_\vv \in\pi_1(X,\vv)$. The class $\sigma_\vv$ is
the image of the positive generator of $\pi_1(\Delta^\ast)$ in $\pi_1(X,\vv)$
associated to any analytic arc $(\Delta,0) \hookrightarrow (\Xbar,P)$ that is
tangent to $\vv$.

The fiber $V_\vv$ of an admissible variation of MHS $\V$ over $X$ at $\vv$ is,
by definition, the limit MHS associated to $\vv$. The underlying complex vector
space is the fiber over $P$ of Deligne's canonical extension of $\V\otimes\cO_X$
to  $\Xbar$. Its rational structure is determined by $\vv$. The local monodromy
operator $h_\vv$ is the image of $\sigma_\vv$ in $\sigma_\vv$ in $\Aut V_\vv$.
It is quasi-unipotent. The local monodromy logarithm $N$ is defined by
$$
N= \frac{1}{k}\log T_\vv^k
$$
where $k>0$ and $T_\vv^k$ is unipotent. The limit MHS $V_\vv$ has two weight
filtrations: $W_\dot$, which is the restriction of the weight filtration of $\V$
to $V_\vv$, and the relative weight filtration $M_\dot$ associated to the local
monodromy logarithm $N:(V_\vv,W_\dot) \to (V_\vv,W_\dot)$. The limit MHS is
$(V_\vv,M_\dot,F^\dot)$, where $F^\dot$ is the limit Hodge filtration. The
filtration $W_\dot$ is a filtration of $V_\vv$ by the mixed Hodge structures
$(W_m V_\vv,F^\dot,M_\dot)$.

The family $\{\cO(\cG_x)\}_{x\in X}$ over $X$ underlies an ind-object of
$\MHS(X,\H)$, \cite{hain:modular,vmhs}. In particular, there is a limit MHS on
$\cO(\cG_\vv)$ with weight filtrations $W_\dot$ and $M_\dot$. These are
respected by the Hopf algebra operations and thus pass to filtrations on
$\cG_\vv$. There is a natural homomorphism
$$
\rho_\vv : \pi_1(X,\vv) \to \cG_\vv(\F).
$$
The universal monodromy logarithm $N_\vv \in \g_\vv$ is the logarithm of the
unipotent part of the Jordan decomposition of $\rho_\vv(\sigma_\vv)$. It
spans a copy of $\Q(1)$ in $\g_\vv$.

A Hodge representation of $\cG_\vv$ is a MHS $(V,F^\dot,M_\dot)$ where $V$ is
in $\Rep(\cG_\vv)$. The coaction
\begin{equation}
\label{eqn:coaction}
V \to V\otimes\cO(\cG_\vv)
\end{equation}
is required to be a morphism of MHS with respect to the weight filtration
$M_\dot$.

In general, the relative weight filtration $M_\dot$ of a filtration preserving
nilpotent endomorphism of a filtered vector space $(V,W_\dot)$ does not
determine the filtration $W_\dot$.\footnote{A simple example where $M_\dot$ and
$N$ do not determine $W_\dot$ is where $V$ is the first homology $H^1(E)$ of an
elliptic curve and $N$ is the monodromy logarithm associated to degenerating to
the nodal cubic. This satisfies $N\neq 0$ and $N^2=0$. The weight filtrations
are
$$
0= W_0 V \subset W_1 V = V \text{ and }
0 = M_{-1}V \subset M_0 V = M_1 V \subset M_2 V = V,
$$
where $M_0 V = \im N$. Apart from the original $W_\dot$, we can take $W_\dot$ to
be $M_\dot$. There is nothing special about this example as $M_\dot$ is always
the relative weight filtration of $N \in M_0\End(V,M_\dot)$.}
However, in the presence of a $\cG_\vv$-action, it does.

\begin{lemma}
\label{lem:wt_filt}
Each Hodge representation $(V,M_\dot)$ of $\cG_\vv$ has a unique filtration
$W_\dot$ that is defined over $\F$, preserved by $N$ and such that $M_\dot$ is
the relative weight filtration of $N \in W_0 \End V$. Moreover, each $W_m V$ is
a sub-MHS of $V$ and the monodromy representation $\cG_\vv \to \Aut V$ preserves
$W_\dot$.
\end{lemma}

\begin{proof}[Sketch of proof]
The coordinate rings $\{\cO(R_x)\}$ of the reductive quotients of the $\cG_x$
form an ind-PVHS over $X$. Denote its fiber over $\vv$ by $\cO(R_\vv)$ and the
corresponding group by $R_\vv$. It is the reductive quotient of $\cG_\vv$. To
simplify notation, we set $\cG = \cG_\vv$ and $R=R_\vv$. Denote the kernel of
the projection $\cG \to R$ by $\U$. It is prounipotent.

We first prove the result in the case where $V$ is an irreducible Hodge
representation of $R$. Note that $V$ is not necessarily irreducible as an
$R$-module. Recall that its weight filtration is denoted by $M_\dot$ and that
$$
N : (V,M_\dot) \to (V,M_\dot)(-1)
$$
is a morphism of MHS. We have to show that there is a unique $m\in \Z$ such that
$M_\dot$ is the weight filtration of $N : V \to V$ shifted by $m$.

Denote the Lie algebra of $R$ by $\r$. Since $\r$ has weight $0$, its relative
weight filtration $M_\dot$ with respect to $N$ is the weight filtration of $N$
and is centered about $0$.  Schmid's work \cite{schmid} implies that there is a
copy of $\sl_2$ in $\r_\R$ that contains $N$ and which has a semi-simple element
$h_0$ whose eigenspaces split the relative weight filtration of $\r_\R$.
Likewise, the relative weight filtration $M_\dot$ of $\cO(R)$ is centered about
0, is the weight filtration of $N$ and is split by the eigenspaces of $h_0$.

The restriction of the monodromy action to $\sl_2$ defines an $\sl_2$ action on
$V_\R$. Decompose $V$ into its $h_0$ eigenspaces:
$$
V_\R = \bigoplus V_r.
$$
Since the coaction $V \to V\otimes\cO(R)$ is a morphism of MHS (with respect to
the weight filtration $M_\dot$), it is strict with respect to $M_\dot$. Since it
is also $\sl_2$-equivariant, it respects the grading by $h_0$ weight.

Say that an $R$-submodule $A$ of $V_\R$ has weight $m_A$ if
$$
N^k : \Gr^M_{m_A+k} A \to \Gr^M_{m_A-k} A
$$
is an isomorphism for all $k$. The weight $m_A$, if it exists, is unique. The
discussion in the previous paragraph implies that every cyclic submodule $A$ of
$V_\R$ that is generated by an $h_0$ eigenvector
$$
v\in (M_k A \setminus M_{k-1}A)\cap V_r
$$
has weight $m_A = k-r$.

Let $\X$ be the set of $R$-submodules of $V_\R$ that have a weight. It is clear
that if $B$ is a non-zero submodule of $A\in \X$, then $B\in \X$ and $m_B=m_A$.
This implies that if $A,B\in \X$ and $A\cap B \neq 0$, then $A+B \in \X$ and
$m_{A+B}=m_A=m_B$. Since $\X$ contains all cyclic modules of the type discussed
above, and since $V$ is a simple Hodge representation of $R$, this implies that
$V\in \X$. The weight filtration
$$
0 = W_{m-1}V \subset W_m V = V,
$$
where $m=m_V$, is defined over $\F$, preserved by $N$, and has the property that
$M_\dot$ is the relative weight filtration of $N\in W_0\End V$. The monodromy
representation $V \to V\otimes \cO(\cG)$ respects $W_\dot$. It is also a
filtration by MHS.

We now consider the general case. Suppose that $V$ is a $\cG$-module. If $\U$
acts trivially, then we are in the semi-simple case above, and so have a weight
filtration. If not, then $\U$ acts unipotently, and so $V^\U\neq 0$. This is an
$R$-module, defined over $\F$, and thus has a weight filtration $W_\dot$ by the
semi-simple case above. Since $V^\U$ is defined over $\F$, the filtration
$W_\dot$ is also defined over $\F$.

Let $m$ be the smallest integer such that $W_m V^\U \neq 0$. This is a sub-MHS
of $V$. The quotient $V/W_m V^\U$ is a Hodge representation of $\cG$ and
$(V/W_m V^\U)^\U$ is a Hodge representation of $R$. Let $k$ be the smallest
integer such that $W_k(V/W_m V^\U)^\U \neq 0$. To complete the proof, it
suffices to show that $m < k$. For if this holds, one can define $W_k V$ to
be the inverse image of $W_k(V/W_m V^\U)^\U$ in $V$. One can then continue
inductively to define the $W_\dot$ filtration of $V$.

To prove that $k>m$, observe that monodromy homomorphism $\U\times V \to V$
induces a homomorphism
\begin{equation}
\label{eqn:gr_monod}
H_1(\u) \otimes W_k(V/W_m V^\U)^\U \to W_m V^\U.
\end{equation}
If the action is trivial, then $W_k(V/W_m V^\U)^\U = \Gr^W_k V^\U$, which
implies that $m<k$. If the action is non-trivial, there is a least $\ell$ such
that the restriction of (\ref{eqn:gr_monod}) to $W_\ell H_1(\u)$ is non-trivial.
The induced map
$$
\Gr^W_\ell \u  \otimes W_k(V/W_m V^\U)^\U \to W_m V^\U
$$
is $R$-invariant, and thus $\sl_2$-invariant. It follows that $k+\ell = m$.
Since $\ell<0$, we have $k>m$. This completes the proof of the existence of a
weight filtration $W_\dot$ with the property that $M_\dot$ is the relative
weight filtration of $N\in W_0\End V$ and where each $W_m V$ is a sub-MHS.
The construction implies that each $W_mV$ is a $\cG$-submodule of $V$ and that
the monodromy representation $\u \times V \to V$ preserves $W_\dot$. Together
these imply that $\cG \times V \to V$ does too.
\end{proof}

\begin{proof}[Proof of Theorem~\ref{thm:vmhs} for tangential base points]
First suppose that $\V$ is an object of $\MHS(X,\H)$ with fiber
$(V,F^\dot,W_\dot,M_\dot)$ over $\vv$. Denote by $\bO$ the local system over $X$
whose fiber over $x\in X$ is $\cO(\cG_x)$. The monodromy coactions
$$
V_x \to V_x \otimes \cO(\cG_x)
$$
give a flat section of the variation $\Hom_\F(\V,\V\otimes\bO)$.
Theorem~\ref{thm:vmhs} implies that this section is in
$$
\G H^0(X,\Hom_\F(\V,\V\otimes\bO)).
$$
The Theorem of the Fixed Part implies that
$$
\G H^0(X,\Hom_\F(\V,\V\otimes\bO))
= F^0W_0 \Hom_{\cG_b}(V_b,V_b\otimes\cO(\cG_b))
$$
for all base points $b$ of $X$, tangential or standard. This implies that the
fiber $V_\vv$ of $\V$ over $\vv$ is in $\HRep(\cG_\vv)$.

Now suppose that $V$ is in $\HRep(\cG_\vv)$. Let $\V$ be the corresponding $\F$
local system over $X$. We have to prove that $\V$ underlies an admissible
variation of MHS whose fiber over $\vv$ is $V$. Lemma~\ref{lem:wt_filt} implies
that $V$ has a weight filtration $W_\dot$ by $\cG_\vv$ submodules such that
$M_\dot$ is the relative weight filtration of $N\in W_0 \End(V)$. This implies
that the local system $\V$ has a filtration $W_\dot\V$ by local systems. Denote
the fiber of $\V$ over $x\in X$ (possibly tangential) by $V_x$. We have to
construct a MHS on each $V_x$.

To do this, we use the the local system  $\bO_\vv$ over $X$ whose fiber over $x$
is $\cO(\cG_{x,\vv})$. This is an ind-object of $\MHS(X;\H)$. The key
observation is that for all $x$, the image of the coaction
$$
V_x \hookrightarrow V \otimes \cO_{x,\vv}
$$
is a mixed Hodge structure. This is true because the image of $V_x$ in
$V\otimes\cO_{\vv,x}$ is the kernel of
$$
\Delta_V\otimes 1 - 1\otimes \Delta :
V\otimes \cO_{\vv,x} \to V \otimes \cO_{\vv,\vv}\otimes \cO_{\vv,x},
$$
which has a MHS as both $\Delta_V$ and $\Delta$ are morphisms of MHS. From this
it follows that $\V$ is isomorphic to the kernel of
$$
\Delta_V\otimes 1 - 1\otimes \Delta :
V\otimes \bO_{\vv} \to V \otimes \bO\otimes \bO_{\vv}
$$
and is therefore an admissible variation of MHS.
\end{proof}

\section{The Natural Homomorphism}
\label{sec:nat_homom}

A homomorphism $\G \to G(\F)$ from a discrete group $\G$ into the $\F$-points of
an affine $\F$-group $G$ induces a functor $\Rep(G) \to \Rep_\F(\G)$ and
therefore a homomorphism
$$
H^\dot(G,V) \to H^\dot(\G,V)
$$
for each $V$ in $\Rep(G)$. When $\G$ is the fundamental group of a (sufficiently
nice) topological space $X$ and $\V$ is a locally constant sheaf over $X$
corresponding to the $\G$-module $V$, there are natural homomorphisms
$$
H^\dot(G,V) \to H^\dot(\G,V) \to H^\dot(X,\V).
$$

Now suppose that $\cG$ is the completion of $\G$ relative to $\rho : \G \to
R(\F)$. The following result is standard. (Cf.\ \cite{hain:malcev}.)

\begin{proposition}
\label{prop:low_deg}
For all $R$-modules $V$, the natural homomorphism
\begin{equation}
\label{eqn:nat_homom}
H^\dot(\cG,V) \to H^\dot(\G,V)
\end{equation}
is an isomorphism in degrees $\le 1$ and injective in degree $2$. If $\G$ is
free, then (\ref{eqn:nat_homom}) is an isomorphism in all degrees.
\end{proposition}

Applying this to the situation where $X$ is a smooth complex algebraic variety,
$\H$ is an $\F$-PVHS over $X$ and $\cG_x$ is the corresponding relative
completion of $\pi_1(X,x)$, we see that for each $\V$ in $\MHS_\F(X,\H)$, there
is a natural homomorphism
\begin{equation}
\label{eqn:morphism}
H^\dot(\cG_x,V_x) \to H^\dot(X,\V).
\end{equation}
It is compatible with products.

\begin{proposition}
\label{prop:nat_homom}
For all objects $\V$ of $\VMHS(X,\H)$, the natural homomorphism
(\ref{eqn:morphism}) is a morphism of MHS. If $X$ is an affine curve, it is
an isomorphism.
\end{proposition}

\begin{proof}
The group $H^\dot(X,\V)$ can be computed by taking an injective resolution $V_x
\hookrightarrow I^\dot$ of the fiber $V_x$ over $x$ in $\Rep(\cG_x)$, taking the
corresponding resolution $\V \hookrightarrow \I^\dot$ of $\V$, and then taking
the cohomology of the total complex of the double complex
$$
E^{\dot}(X,\I^\dot) := \bigoplus_{m\ge 0} E^\dot(X,\I^m)
$$
where $E^\dot(X,\I^m)$ denotes the complex of smooth forms with coefficients in
$\I^m$. The homomorphism (\ref{eqn:morphism}) is induced by the inclusion
$(I^\dot)^{\cG_x} \hookrightarrow E^0(X,\I^\dot) \hookrightarrow
E^{\dot}(X,\I^\dot)$. To prove the result, we need to realize this construction
using mixed Hodge complexes.

Choose a relatively injective resolution $V_x \to I^\dot$ of $V_x$ in
$\HRep(\cG_x)$. Theorem~\ref{thm:vmhs} implies that this corresponds to a
resolution $\V \to \I^\dot$ in ind-$\MHS(X,\H)$ with fiber $I^\dot$ over $x$.

Write $X=\Xbar-D$, where $\Xbar$ is smooth and complete, and where $D$ is a
divisor with normal crossings. Denote by $K^\dot(\Xbar,D;\bA)$ Saito's mixed
Hodge complex \cite{saito:mhc} that computes the MHS on $H^\dot(X,\bA)$, where
$\bA$ is an admissible variation of MHS over $X$. When $X$ is a curve, one can
use Zucker's MHC \cite{zucker} instead.

The inclusion $\V \hookrightarrow \I^\dot$ induces a morphism of mixed Hodge
complexes
$$
K^\dot(\Xbar,D;\V) \to \tot K^\dot(\Xbar,D;\I^\dot)
$$
and thus an isomorphism of MHS on cohomology. It is a quasi-isomorphism as $\V
\to \I^\dot$ is a resolution. The theorem of the fixed part implies that
$H^0(X,\I^\dot)$ is a mixed Hodge complex and that its inclusion into $\tot
K^\dot(\Xbar,D;\I^\dot)$ is a morphism. This inclusion, coupled with the natural
isomorphisms
$$
H^\dot(\cG_x,V_x) \cong H^\dot((I^\dot)^{\cG_x}) \cong H^\dot(H^0(X,\I^\dot))
$$
and
$$
H^\dot(X,\V) \cong H^\dot(K^\dot(\Xbar,D;\V))
\cong H^\dot(\tot K^\dot(\Xbar,D;\I^\dot)),
$$
induces the natural morphism (\ref{eqn:morphism}). It is a morphism of MHS as it
is induced by a morphism of mixed Hodge complexes.

The last assertion follows from Proposition~\ref{prop:low_deg} as every affine
curve is a model of the classifying space of a free group.
\end{proof}

\section{Comparison with Deligne--Beilinson Cohomology}
\label{sec:db_coho}

As in the previous section, $X$ is the complement of a normal crossing divisor
$D$ in a smooth projective variety $\Xbar$. Saito's mixed Hodge complex
\cite{saito:mhc} with coefficients in an admissible variation of MHS $\V$ over
$X$ will be denoted by $K^\dot(\Xbar,D;\V)$.

Suppose that $(K^\dot,W_\dot)$ is a filtered complex, where $W_\dot$ is
increasing. Recall from \cite[1.3.3]{deligne:hodge2} that the filtered complex
$\Dec_W K^\dot$ is the same complex $K^\dot$ endowed with the shifted filtration
({\em filtration decal\'ee}) defined by
$$
W_m \Dec_W K^j = \{k \in W_{m+j} K^j : dk \in W_{m+j+1}K^{j+1}\}.
$$
The reason for shifting the filtration $W_\dot$ is that if $K^\dot$ is a mixed
Hodge complex, then the inclusion $W_m \Dec_W K^\dot \hookrightarrow K^\dot$
induces isomorphisms
$$
H^\dot(W_m\Dec_W K^\dot) \overset{\simeq}{\longrightarrow} W_m H^\dot(K^\dot)
,\quad
H^\dot(F^pW_m\Dec_W K^\dot) \overset{\simeq}{\longrightarrow}
F^pW_m H^\dot(K^\dot).
$$

\begin{definition}
The Deligne--Beilinson (DB) cohomology $H^\dot_\cD(X,\V)$ is defined to be the
cohomology of the complex
\begin{multline}
\label{eqn:db_coho_def}
\cone\big(F^0W_0\Dec_W K^\dot_\C(\Xbar,D;\V) \oplus
W_0\Dec_W K^\dot_\Q(\Xbar,D;\V)
\cr
\to W_0\Dec_W K^\dot_\C(\Xbar,D;\V)\big)[-1].
\end{multline}
\end{definition}

Standard arguments imply that it is independent of the choice of the
compactification $\Xbar$ and that it can be expressed as the extension
\begin{equation}
\label{eqn:db-ses}
0 \to \Ext^1_{\MHS_\F}\big(\Q,H^{j-1}(X,\V)\big) \to H_\cD^j(X,\V)
\to \G H^j(X,\V) \to 0.
\end{equation}

\subsection{Comparison of DB-cohomologies}
\label{sec:comparison}

Suppose that $\H$ is an $\F$-PVHS over $X$ and that $\cG_x$ is the corresponding
relative completion of $\pi_1(X,x)$.

\begin{theorem}
\label{thm:comp}
For all objects $\V$ of $\VMHS(X,\H)$, there is a natural homomorphism
$$
H^\dot_\cD(\cG_x,V_x) \to H^\dot_\cD(X,\V)
$$
which respects the products and whose degree $j$ part fits into a commutative
diagram
$$
\xymatrix{
0 \ar[r] & \Ext^1_\MHS(\F,H^{j-1}(\cG_x,V_x)) \ar[r]\ar[d] &
H_\cD^j(\cG_x,V_x) \ar[r]\ar[d] & \G H^j(\cG_x,V_x) \ar[r]\ar[d] & 0 \cr
0 \ar[r] & \Ext^1_\MHS(\F,H^{j-1}(X,\V)) \ar[r] &
H_\cD^j(X,\V) \ar[r] & \G H^j(X,\V) \ar[r] & 0
}
$$
where the right and left hand vertical maps are induced by the natural
homomorphism (\ref{eqn:morphism}).
\end{theorem}

\begin{proof}
Choose a relatively injective resolution $V_x \hookrightarrow I^\dot$ of $V_x$
in $\HRep(\cG_x)$. The corresponding resolution of $\V$ in $\MHS(X,\H)$ will be
denoted by $\V \hookrightarrow \I^\dot$.

Define
$$
\Dec_W \tot K^\dot(\Xbar,D;\I^\dot) =
\bigoplus_{n\ge 0} \Dec_W K^\dot(\Xbar,D;\I^n).
$$
In other words, use only the degree in $K^\dot$ to compute $\Dec_W$, and ignore
the degree in $\I^\dot$. With this convention, the inclusion above induces an
inclusion
$$
(I^\dot)^{\cG_x} = H^0(X,\I^\dot) \hookrightarrow
\Dec_W \tot K^\dot(\Xbar,D;\I^\dot)
$$
that preserves the Hodge and weight filtrations. This inclusion induces a chain
map from the complex (\ref{eqn:dbcoho-group}) to the complex
(\ref{eqn:db_coho_def}). This induces the morphism on Deligne cohomology that
is compatible with products.
\end{proof}

Combining this with the second assertion of Proposition~\ref{prop:nat_homom}, we
conclude:

\begin{corollary}
\label{cor:curve_isom}
If $X$ is an affine curve, then the natural map
$
H^\dot_\cD(\cG_x,V_x) \to H^\dot_\cD(X,\V)
$
is an isomorphism for all $\V$ in $\MHS(X,\H)$. It is compatible with products.
\end{corollary}

\subsection{Extensions of variations of MHS}

Denote the category of admissible variations of $\F$ mixed Hodge structure over
$X$ by $\MHS_\F(X)$. This is the ``union'' of the categories $\MHS_\F(X,\H)$
over all semi-simple objects $\H$ of $\MHS_\F(X)$.

\begin{lemma}
\label{lem:lim}
For all $\V$ in $\MHS_\F(X)$, the inclusions $\MHS(X,\H) \hookrightarrow
\MHS_\F(X)$ induce an isomorphism
$$
\varinjlim_\H \Ext_{\MHS(X,\H)}^\dot(\F,\V) \to \Ext_{\MHS_\F(X)}^\dot(\F,\V).
$$
where $\H$ ranges over a set of representatives of the isomorphism classes
of semi-simple objects of $\MHS_\F(X)$ that contain $\V$.
\end{lemma}

\begin{proof}
This follows from the fact that every finite diagram of objects in $\MHS_\F(X)$
lies in $\MHS(X,\H)$ for some semi-simple object $\H$ of $\MHS_\F(X)$.
\end{proof}

Theorem~\ref{thm:big} is obtained by combining the following two results.

\begin{theorem}
\label{thm:big_thm}
If $X$ is a quasi-projective manifold and $\V$ is an admissible variation of MHS
over $X$, there there is a homomorphism
$$
\Ext^\dot_{\MHS(X)}(\F,\V) \to H^\dot_\cD(X,\V)
$$
which is compatible with products. It is an isomorphism in degrees $\le 1$ and
injective in degree $2$.
\end{theorem}

\begin{proof}
The first assertion follows directly from Lemma~\ref{lem:lim},
Corollary~\ref{cor:vmhs} and Theorem~\ref{thm:comp}. The last assertion follows
by applying Proposition~\ref{prop:nat_homom} to the map from the exact sequence
(\ref{eqn:ses_ext}) with $G=\cG_x$ to the exact sequence (\ref{eqn:db-ses}).
\end{proof}

Combining this with Corollary~\ref{cor:curve_isom} we conclude:

\begin{corollary}
If $X$ is an affine curve, then for all admissible variations of $\F$-MHS $\V$
over $X$, there is a natural isomorphism
$$
\Ext^\dot_{\MHS_\F(X)}(\F,\V) \cong H^\dot_\cD(X,\V)
$$
which is compatible with products.
\end{corollary}

\section{Brown's Computation of the Cup Products of Eisenstein Series}
\label{sec:brown}

In this section, we briefly explain how Brown's computations
\cite[\S11]{brown:mmv} can be interpreted in terms of Deligne cohomology of the
relative completion of $\SL_2(\Z)$. This discussion is not self contained; the
reader will need some familiarity with \cite{brown:mmv}, \cite{hain:modular} and
\cite{mem}. 

\begin{remark}
The reader should be aware that, whereas we use {\em left} group cochains, Brown
uses {\em right} cochains. He also multiplies paths in the topologist's order,
the opposite of that used here. So, with the conventions in this paper, his
fundamental group of a space is the opposite group of the fundamental group used
here. His work coincides with what we have here once one makes the following
observations:

Denote the opposite group of a group $\G$ by $\G^\op$. Every left $\G$ module
$V$ is naturally a right $\G^\op$ module, denoted $V^\op$. Suppose that $V$ is a
left $\G$-module. The function that takes a left $\G$-invariant function $f :
\G^{n+1} \to V$ to the right invariant function $F : (\G^\op)^{n+1} \to V^\op$
defined by
$$
F(g_0,\dots,g_n) = f(g_0^{-1},\dots,g_n^{-1})
$$
is an isomorphism and induces a chain isomorphism
$$
\cC^\dot(\G,V) \to \cC^\dot(\G^\op,V^\op).
$$
\end{remark}

\subsection{Preliminaries}

Let $\G=\SL_2(\Z)$. We identify this with the orbifold fundamental group
$\pi_1(\M_{1,1},\partial/\partial q)$ of the modular curve with base point the
natural tangent vector at the cusp. Let $\cG$ be the completion (with $\F=\Q$)
of $\G$ with respect to the inclusion $\G\hookrightarrow \SL_2(\Q)$. It has a
natural (limit) MHS with weight filtrations $W_\dot$ (the ``global'' weight
filtration) and $M_\dot$ (the relative weight filtration). The defining
representation of $\SL_2$ will be denoted by $H$. The corresponding local system
$\H$ over $\M_{1,1}$ underlies a polarized variation of Hodge structure of
weight $+1$. Its fiber over $\partial/\partial q$ is $H$. It has a natural limit
MHS with weight filtrations $W_\dot$ and $M_\dot$. It is a Hodge representation
of $\cG$.

Brown's computations are most naturally interpreted in the Deligne cohomology of
$\cG$. Theorem~\ref{thm:comp} implies that the natural map
$$
H^\dot_\cD(\cG,S^m H(r)) \to H^\dot_\cD(\M_{1,1},S^m \H(r))
$$
is an isomorphism. These groups vanish when $m$ is odd.

For all $r\in \Z$, the homomorphism from $\cG$ to the trivial group induces a
natural isomorphism
$$
H^j_\cD(\cG,S^{m}H(r)) \cong \Ext^j_\MHS(\Q,\Q(r)).
$$
This group vanishes, except when $r=j=0$ and when $j=1$ and $r>0$. When $m>0$,
standard results (Eichler-Shimura, Manin-Drinfeld) about the cohomology of
modular curves imply that
$$
H^1_\cD(\cG,S^{m}H(r)) \cong
\begin{cases}
\Q & m=2n,\ r = 2n+1,\cr
0 & \text{otherwise}
\end{cases}
$$
and that
$$
H^2_\cD(\cG,S^{m}\H(r)) \cong \Ext^1_\MHS(\Q,H^1(\M_{1,1},S^{m}\H(r))).
$$
This is non-zero only when $m=2n>0$.

We will describe the image of the cup product in {\em real} Deligne cohomology
as in degree 2 it is more easily described:
\begin{align*}
H^2_\cD(\cG,S^{2n}\H_\R(r)) &\cong
\Ext^1_{\MHS_\R}(\R,H^1(\M_{1,1},S^{2n}\H(r)))
\cr
&= \Ext^1_{\MHS_\R}(\R,\R(r-2n-1)) \oplus
\bigoplus_{f\in \B_{2n+2}}\Ext^1_{\MHS_\R}(\R,V_f(r))
\cr
&\cong
\begin{cases}
\R \oplus \bigoplus_{f\in \B_{2n+2}} V_{f,\R} & n > 0,\ r\ge 2n+2,\cr
0 & \text{otherwise.}
\end{cases}
\end{align*}
Here $\B_{2n+2}$ denotes the set of normalized Hecke eigen cusp forms of weight
$2n+2$ and $V_f$ the 2-dimensional real Hodge structure corresponding to
$f\in\B_{2n+2}$.\footnote{The projection $\pi_f$ is defined on cohomology with
$\F = \Q_f$, the field generated by the Fourier coefficients of $f$.}

The cup product is most easily described by giving its projections onto
the components of the degree 2 cohomology corresponding to Hecke eigenforms.
To this end, recall that, when $\F = \Q$ and $\R$, one has projections
\begin{multline}
\label{eqn:projnsQ}
\pi_\eis : H^2_\cD(\cG,S^{2n}\H_\F(r)) \to \Ext^1_\MHS(\F,\F(r-2n-1))
\cr
\text{ and }
\pi_\cusp : H^2_\cD(\cG,S^{2n}\H_\F(r)) \to
\Ext^1_\MHS(\F,H^1_\cusp(\M_{1,1},S^{2n}\H_\F(r))).
\end{multline}
onto the Eisenstein and cuspidal parts, respectively. When $\F=\R$, we can
further decompose $\pi_\cusp$ into the projections
\begin{equation}
\label{eqn:projnsR}
\pi_f : H^2_\cD(\cG,S^{2n}\H_\R(r)) \to \Ext^1_\MHS(\R,V_f(r)) \cong V_{f,\R},
\end{equation}
where $f \in \B_{2n+2}$.

To describe the cup product, we also have to decompose the coefficients into
their irreducible pieces. Suppose that $n=j+k$, with $j>k>0$. The standard
isomorphism
$$
S^{2j}H\otimes S^{2k}H \cong \bigoplus_{r=0}^k S^{2n-2r}H(-r)
$$
of $\SL(H)$-modules respects the mixed Hodge structures. Fix a normalization
of the $\SL(H)$-invariant projection
$$
p_r : S^{2j}H\otimes S^{2k}H \to S^{2n-2r}H(-r).
$$
such as the one given in \cite[\S2.4]{brown:mmv}. When $r \le 2\min\{j,k\}$, one
can compose the cup product
$$
H^1_\cD(\cG,S^{2j}H(2j+1)) \otimes H^1_\cD(\cG,S^{2k}H(2k+1))
\to H^2_\cD(\cG,S^{2j}H\otimes S^{2k}H(2n+2))
$$
with the projection $p_r$ and the projections (\ref{eqn:projnsQ})  and
(\ref{eqn:projnsR}) to obtain maps
\begin{equation}
\label{eqn:cup_eis}
\Phi_r^\eis : H^1_\cD(\cG,S^{2j}H(2j+1)) \otimes H^1_\cD(\cG,S^{2k}H(2k+1))
\to \Ext^1_\MHS(\R,\R(2n-r+2)).
\end{equation}
and
\begin{equation}
\label{eqn:cup_cusp}
\Phi_{r,f} : H^1_\cD(\cG,S^{2j}H(2j+1)) \otimes H^1_\cD(\cG,S^{2k}H(2k+1))
\to V_{f,\R}
\end{equation}
for each $f\in \B_{2n-2r+2}$. Brown \cite[\S10, \S11]{brown:mmv} shows that
$\Phi_r^\eis$ vanishes except in the extremal case where $r=2\min\{j,k\}$.

Note that, since the generators of $H^1_\cD(\cG,S^{2n}\H(2n+1))$ are invariant
under the de~Rham involution $\Frbar_\infty$, so is the image of the cup
product.

\subsection{Modular forms and $\cO(\cG)$}

Denote the canonical extension of the flat bundle $\H\otimes\cO_{\M_{1,1}}$ to
$\Mbar_{1,1}$ by $\cH$. Denote the upper half plane by $\h$. Set $q = e^{2\pi i
\tau}$. The pullback of $\cH$ to the $q$-disk is a trivial holomorphic vector
bundle. The trivializing sections are
$$
\aa \text{ and }\bw= -2\pi i(\bb - \tau \aa),
$$
where $\aa$ and $\bb$ is the standard basis of $H$, defined in
\cite[\S9]{hain:modular}. They can also be regarded as trivializing sections of
$\H_\Q$ over $\h$.

For a modular form $f : \h \to \C$ of $\SL_2(\Z)$ of weight $2n+2$, set
$$
\w_f = 2\pi i f(\tau) \be(\tau) d \tau.
$$
where $\be(\tau) = \bw^{2n} = (2\pi i)^{2n}(\bb - \tau \aa)^{2n}$. This is a
$\G$-invariant 1-form on $\h$ with coefficients in $\cH$. When $f$ is a
normalized eigenform, it is section of $\Omega^1_{\Mbar_{1,1/\Q}}(\log P)\otimes
F^{2n}S^{2n}\cH$ defined over $\Q$, \cite[\S19]{hain:kzb}. The de Rham theorem
for relative completion \cite{hain:malcev} implies that if $f_j$ is a modular
form of weight $2n_j+2$, then the function
$$
I_{[f_1|\dots|f_m]} : 
\SL_2(\Z) \to S^{n_1}H \otimes S^{n_2} H \otimes \cdots \otimes S^{n_m} H
$$
defined by the {\em regularized} (cf.\ \cite{brown:mmv}) iterated integral
$$
I_{[f_1|\dots|f_m]} :
\gamma \mapsto \int_\gamma \w_{f_1} \w_{f_2} \dots \w_{f_m}
$$
is an element of
$
F^{2n+m}\big(
\cO(\cG) \otimes S^{n_1}H \otimes S^{n_2} H \otimes \cdots \otimes S^{n_m} H
\big)
$
where $n=n_1+\dots+n_m$. It is a 1-cochain satisfying
$$
I_{[f_1|\dots|f_m]} \in F^{2n+m}
\cC^1(\cG, S^{n_1}H \otimes S^{n_2} H \otimes \cdots \otimes S^{n_m} H).
$$
When $m=1$, it is a cocycle. Invariance under $\G$ and standard properties of
iterated integrals imply that
$$
I_{[f_1]} \cup I_{[f_2]} + \delta I_{[f_1|f_2]} = 0.
$$
See \cite[(5.2)]{brown:mmv}

\begin{remark}
Note that in \cite{brown:mmv}, $\SL_2(\Z)$ acts on the {\em right} of the vector
space $V = \Q X \oplus \Q Y$ by the action
$$
\begin{pmatrix} X\cr Y\end{pmatrix} \mapsto
\begin{pmatrix} a & b \cr c & d \end{pmatrix}
\begin{pmatrix} X\cr Y\end{pmatrix}.
$$
This corresponds to the left action
$$
\begin{pmatrix} \aa & -\bb \end{pmatrix} \mapsto
\begin{pmatrix} \aa & -\bb \end{pmatrix}
\begin{pmatrix} a & b \cr c & d \end{pmatrix}
$$
of $\SL_2(\Z)$ on $H = \Q\aa \oplus \Q\bb$ under the identification $X=\bb$,
$Y=\aa$.
\end{remark}

\subsection{Interlude}

For clarity, we make a few remarks about the Betti and de~Rham structures of the
coefficients $S^m H(r)$. Although we are not using the $\Q$-DR structure in the
construction of DB-cohomology, we describe the $\Q$-DR structure of $S^m H(r)$
as that should help clarify the normalizations we choose in this section and the
next as well as its role in applications, such as those in \cite{mem}. Relevant
background can be found in \cite{hain:modular}.

Suppose that $V$ is a Hodge structure with underlying $\Q$-Betti and
$\Q$-de~Rham spaces $V^B$ and $V^\DR$. One can identify the $\Q$-de~Rham
space of $V(r)$ with $V^\DR$ and its $\Q$-Betti space with
$$
V(r)^B := (2 \pi i)^r V^B.
$$
We now explain how this works when $V=S^m H$.

As a Hodge structure of weight 1, $H^B$ has $\Q$ basis $\aa$ and $\bb$ and
$H^\DR$ has a $\Q$-basis $\aa$ and $\bw = -2\pi i \bb$. The vector $\bw$ spans
$F^1 H^\DR$. The $m$th symmetric power $S^m H$ thus has $\Q$-Betti basis the
monomials $\aa^j\bb^{m-j}$, and $\Q$-de Rham basis $(2\pi
i)^{m-j}\aa^j\bb^{m-j}$. It decomposes as a sum
$$
S^m H = \Q(-m) \oplus \Q(-m+1) \dots \oplus \Q(1) \oplus \Q(0).
$$
The $\Q$-Betti space of $S^m H(r)$ is $S^mH(r)^B = (2\pi i)^r S^m H^B$. For
future reference, the $\Q$-de~Rham and $\Q$-Betti bases of $S^m H(r)$  are given
in Figure~\ref{fig:bases}.

\begin{figure}[!ht]
$$
\begin{array}{|c||c|c|c|c|c|}
\hline
 & \Q(r) & \dots & \Q(r+j) & \dots & \Q(m+r) \\
\hline
\DR & \bw^m & \phantom{xxxxxxx} & \bw^{m-j}\aa^j & \phantom{xxxxxxx} & \aa^m \\
B & (2\pi i)^{r}\bb^m && (2\pi i)^{r}\bb^{m-j}\aa^j& & (2\pi i)^r \aa^m \\
\hline
\end{array}
$$
\caption{$\Q$-de~Rham and $\Q$-Betti bases of $S^m H(m+r)$}
\label{fig:bases}
\end{figure}

Denote the {\em normalized} Eisenstein series of weight $2k$ by $G_{2k}$. (The
definition is recalled below.) If we regard $\w_{G_{2k}}$ as taking values in
$S^{2k-2}H(2k-1)$, then
$$
I_{[G_{2m_1}|\dots|G_{2m_r}]} \in
F^0 W_0 \big[
\cO(\cG)\otimes S^{2m_1-2}H\otimes \cdots \otimes S^{2m_r-2}H(2m-r)
\big],
$$
where $m=m_1+\dots+m_r$.

\subsection{The Eisenstein cocycle}

The generator $\bG_{2n+2}$ of $H^1_\cD(\cG,S^{2n}H(2n+1))$ corresponds to the
normalized Eisenstein series
$$
G_{2n}(q) = -\frac{B_{2n}}{4n} + \sum_{n=1}^\infty \sigma_{2n-1}(n)q^n
$$
of weight $2n+2$, where $B_{2n}$ denotes the $2n$th Bernoulli number. The
$\Q$-DR 1-cocycle that corresponds to $\w_{G_{2n+2}}$ is $I_{[G_{2n+2}]}$.
Brown's computation \cite[Lem.~7.1]{brown:mmv} implies that
$$
Z_{2n} :=
\begin{bmatrix}
& -\frac{(2n)!}{2}\zeta(2n+1)\aa^{2n} \cr
I_{[G_{2n+2}]}(\bb,\aa) && (2\pi i)^{2n+1} e_{2n+2}^0(\bb,\aa)
\end{bmatrix}
$$
is a 1-cocycle that represents $\bG_{2n}$, where $e^0_{2n+2}$ is the 1-cocycle
defined in \cite[\S7.3]{brown:mmv}. It is the ``rational part'' of the
normalized period polynomial of $G_{2n+2}/(2\pi i)^{2n+1}$. Since $\G \to
\cG(\Q)$ is Zariski dense, the restriction map $\cC^\dot(\cG,V) \to
\cC^\dot(\G,V)$ on group cochains is injective. Since $I_{[G_{2n+2}]}$ and
$\delta \bb^{2n} \in \cC^1(\cG_\Q,S^{2n}H)$, $e_{2n+2}^0$ is also in
$\cC^1(\cG,S^{2n}H)$. It is rational as it takes rational values on $\G$.

\begin{remark}
The class $\bG_{2n+2}$ restricts to the base point $\partial/\partial q$
to give an element of
$$
\Ext_\MHS^1(\Q,S^{2n}H(2n+1)) \cong \oplus_{r=1}^{2n+1} \Ext^1_\MHS(\Q,\Q(r)).
$$
This is the fiber of the corresponding VMHS over $\M_{1,1}$ at
$\partial/\partial q$ and is the limit MHS of this variation with respect to the
tangent vector $\partial/\partial q$. Brown's cocycle implies that this limit
MHS is the coset of $-(2n)!\zeta(2n+1)\aa^{2n}/2$ in
$$
\Ext^1_\MHS(\Q,S^{2n}H(2n+1)) \cong S^{2n}H_\C/(2\pi i)^{2n+1} S^{2n}H_\Q.
$$
It corresponds to the coset of $-(2n)!\zeta(2n+1)$ in
$\Ext^1_\MHS(\Q,\Q(2n+1))$.
\end{remark}

\subsection{The cup product of Eisenstein classes}

In this section, we compute a formula for the external cup product
$\bG_{2j+2}\cup \bG_{2k+2}$. This will allow us to use Brown's period
computations \cite{brown:mmv} to compute the cup products of the classes
$\bG_{2m}$ in Deligne cohomology.

Suppose that $j \ge k > 0$. Set $n=j+k$.  In the following discussion, the cup
product of groups cochains is to be interpreted as the external cup product
$$
\cup : \cC^\dot(\cG,S^{2j}H(2j+1))\otimes \cC^\dot(\cG,S^{2k}H(2k+1))
\to \cC^\dot(\cG,S^{2j}H\otimes S^{2k}H(2n+2)).
$$
Set $\bB_{2m+2}=
-\frac{(2m)!}{2}\zeta(2m+1)\aa^{2m}$. Observe that
{\small
\begin{multline*}
Z_{2j+2}\cup_{1/2} Z_{2k+2}
\cr
=
\begin{bmatrix}
\xymatrix@C=-36pt@R=8pt{
& (2\pi i)^{2k+1}\bB_{2j+2}\cup e_{2k+2}^0 -
(2\pi i)^{2j+1}e_{2j+2}^0\cup \bB_{2k+2} + C
\cr
I_{[G_{2j+2}]}\cup I_{[G_{2k+2}]}
&&
(2\pi i)^{2n+2} e_{2j+2}^0 \cup e_{2k+2}^0
}
\end{bmatrix}
\end{multline*}
}
with $\delta C = 0$, where
$$
C=
\frac{1}{2}\big(\bB_{2j+2}\cup (\delta \bB_{2k+2})
-(\delta \bB_{2j+2})\cup \bB_{2k+2}\big).
$$
For future reference, we note that this is {\em real}.  Since the coboundary
operator $\delta$ is strict with respect to $W_\dot$, we can find $h\in
W_0\cC^1(\cG_\Q,S^{2j}H_\Q\otimes S^{2k}H_\Q)(2n+2)$ such that
$$
e_{2j+2}^0 \cup e_{2k+2}^0 + \delta h = 0.
$$
(Brown \cite[\S10.3]{brown:mmv} denotes a choice of $h$ by $-e^0_{2j+2,2k+2}$.)
Since
$$
I_{[G_{2j+2}]}\cup I_{[G_{2k+2}]} + \delta I_{[G_{2j+2}|G_{2k+2}]} = 0
$$
and since (after twisting)
$$
I_{[G_{2j+2}|G_{2k+2}]} \in F^0 W_0 \cC^1(\cG,S^{2j+1}H\otimes S^{2k+1}H(2n+2)),
$$
$$
D :=
\begin{bmatrix}
& 0 & 
\cr
I_{[G_{2j+2}|G_{2k+2}]} && (2\pi i)^{2n+2} h
\end{bmatrix}
$$
is a  1-cochain.

The cup product, as an element of $H^2_\cD(\M_{1,1},S^{2j}\H\otimes
S^{2k}\H(2n+2))$, is represented by the cocycle
$$
Z_{2j+2}\cup_{1/2} Z_{2k+2} + \delta D
=
\begin{bmatrix}
& E \cr
0 && 0
\end{bmatrix}
$$
where $E$ is the 1-cocycle
$$
I_{[G_{2j+2}|G_{2k+2}]} +
(2\pi i)^{2k+1}\bB_{2j+2}\cup e_{2k+2}^0 -
(2\pi i)^{2j+1}e_{2j+2}^0\cup \bB_{2k+2} + C - (2\pi i)^{2n+2}h.
$$
with values in $S^{2j}H\otimes S^{2k}H(2n+2)$.

\begin{proposition}
\label{prop:cup}
With the notation above, the cup product $\bG_{2j+2}\cup \bG_{2k+2}$ in the {\em
real} Deligne cohomology group
$$
H^2_\cD(\cG,S^{2j}H_\R\otimes S^{2k}H_\R(2n+2)) \cong
\Ext^1_\MHS\big(\R,H^1(\cG,S^{2j}H_\R\otimes S^{2k}H_\R(2n+2))\big)
$$
is represented by the imaginary part of
$$
\Im(I_{[G_{2j+2}|G_{2k+2}]} +
(2\pi i)^{2k+1}\bB_{2j+2}\cup e_{2k+2}^0 -
(2\pi i)^{2j+1}e_{2j+2}^0\cup \bB_{2k+2}).
$$
\end{proposition}

\begin{proof}
To compute the cup product in {\em real} Deligne cohomology, we can modify $E$
by any element of the form
$$
\delta 
\begin{bmatrix}
& 0 & \cr
0 && z
\end{bmatrix}
$$
where $z$ is any cocycle in $W_0 \cC^1(\cG_\R,S^{2j}H\otimes S^{2k}H(2n+2))$.
Since $C$ and $h$ are real, this implies that
$$
\bG_{2j+2}\cup \bG_{2k+2} \in H^2_\cD(\cG,S^{2j}H_\R\otimes S^{2k}H_\R(2n+2))
\cong H^1(\cG,S^{2j}H_\R\otimes S^{2k}H_\R(2n+2))
$$
is represented by the imaginary part
$$
\Im(I_{[G_{2j+2}|G_{2k+2}]} +
(2\pi i)^{2k+1}\bB_{2j+2}\cup e_{2k+2}^0 -
(2\pi i)^{2j+1}e_{2j+2}^0\cup \bB_{2k+2})
$$
of $E$ with respect to the isomorphisms
\begin{multline*}
\Ext^1_\MHS\big(\R,H^1(\cG,S^{2j}H_\R\otimes S^{2k}H_\R(2n+2))\big)
\cr
\cong H^1(\cG,S^{2j}H_\C\otimes S^{2k}H_\C)/(2\pi i)^{2n+2}
H^1(\cG,S^{2j}H_\R\otimes S^{2k}H_\R)
\cr
\cong iH^1(\cG,S^{2j}H_\R\otimes S^{2k}H_\R).\hspace*{2in}
\end{multline*}
\end{proof}

\subsection{The cuspidal projection of the cup product}

In this section, we interpret Brown's period computations
\cite[\S11]{brown:mmv} in terms of Deligne cohomology.

Recall that the {\em modular symbol} of a cusp form $f$ of weight $2n+2$
is the homogeneous polynomial
$$
\br_f := \int_0^\infty \w_f \in S^{2n}H_\C,
$$
where the path of integration is the imaginary axis. When $f$ has real
Fourier coefficients (such as when $f\in \B_{2n+2}$), this decomposes
$$
\br_f = \br_f^+ + i\br_f^-
$$
where $\br_f^\pm$ are real and where all terms of $\br_f^+$ (resp.\ $\br_f^-$)
have even (resp.\ odd) degrees in $\aa$ and $\bb$. The polynomials $\br_f^\pm$
represent classes in $H^1_\cusp(\M_{1,1},S^{2n}H_\R)$.

\begin{theorem}[Brown]
Suppose that $j \ge k > 0$ and that $0 \le r \le 2k$. Set $n=j+k$. If $f$ is a
normalized Hecke eigen cusp form of weight $2n-2r+2$, then
$$
\Phi_{r,f}(\bG_{2j+2}\otimes \bG_{2k+2}) =
(2\pi i)^r \Lambda(f,2n-r+2) \br_f^\epsilon
$$
where $\Lambda(f,s) = (2\pi)^{-s}\G(s)L(f,s)$ is the completed $L$-function
associated to $f$ and $\epsilon \in \{\pm\}$ is the sign of $(-1)^{r+1}$.
\end{theorem}

\begin{proof}
This follows directly from Equations (11.2) and (11.4) and Corollary~11.2 in
\cite{brown:mmv} and Proposition~\ref{prop:cup}.
\end{proof}

\subsection{An external cup product}

We can also use Brown's cocycle $\bG_{2n+2}$ to compute the cup product
$$
\Ext^1_\MHS(\Q,\Q(2m-1)) \otimes H^1_\cD(\cG,S^{2n}H(2n+1)) \to 
H^2_\cD(\cG,S^{2n}H(2n+2m))
$$
of the class of $\zeta(2m-1)$ with the class $G_{2n+1}$ of the Eisenstein
series, where we are identifying $H^1_\cD(\cG,\Q(2m-1))$ with
$\Ext^1_\MHS(\Q,\Q(2m-1))$ via the homomorphism from $\cG$ to the trivial group.

%% Here we are thinking of $\Ext^1_\MHS(\Q,\Q(2n+1))$ as the Deligne
%% cohomology $H^1_\cD({\pmb 1},\Q(2n+1))$ of the trivial group.

One has the projection
$$
\pi_\cusp : H^2_\cD(\cG,S^{2n}H_\Q(2n+2m))
\to \Ext^1_\MHS(\Q,H^1_\cusp(\M_{1,1},S^{2n}\H(2n+2m)))
$$
onto its cuspidal part. One also has the Eisenstein projection
$$
\id - \pi_\cusp :  H^2_\cD(\cG,S^{2n}H_\Q(2n+2m))
\to \Ext^1_\MHS(\Q,\Q(2m-1)).
$$

Denote the class of $\zeta(2m-1)$ in $\Ext^1_\MHS(\Q,\Q(2m-1))$ by $\bZ_{2m-1}$.
It is represented by the Deligne cocycle
$$
\begin{bmatrix}
& \zeta(2m-1) \cr 0 && 0
\end{bmatrix}.
$$

\begin{proposition}
\label{prop:external_cup}
The cuspidal projection of
$$
\bZ_{2m-1} \cup \bG_{2n+2} \in H^2_\cD(\cG,S^{2n}H(2n+2m))
$$
is trivial. Its Eisenstein projection is
$$
\bZ_{2m-1} \in \Ext^1_\MHS(\Q,\Q(2m-1)).
$$
\end{proposition}

\begin{proof}
The cup product is represented by
\begin{align*}
&\phantom{xx}\begin{bmatrix}
& \zeta(2m-1)\cr 0 && 0
\end{bmatrix}
\cup_0
\begin{bmatrix}
\xymatrix@C=-24pt@R=0pt{
& -\frac{(2n)!}{2}\zeta(2n+1)\aa^{2n} \cr
I_{[G_{2n+2}]}(\bb,\aa) && (2\pi i)^{2n+1} e_{2n+2}^0(\bb,\aa)
}
\end{bmatrix}
\cr
&=
\begin{bmatrix}
\xymatrix@C=0pt@R=0pt{
& (2\pi i)^{2n+1}\zeta(2m-1)e_{2n+2}^0(\bb,\aa)
\cr
0 & & 0
}
\end{bmatrix}
\end{align*}
This has trivial projection under $\pi_\cusp$. The top entry lies in $\C
e_{2n+2}^0(\bb,\aa)$ and is well-defined mod $(2\pi
i)^{2m+2n}e_{2n+2}^0(\bb,\aa)$. It therefore corresponds to $\zeta(2m-1)$ mod
$(2\pi i)^{2m-1}\Q$ in $\Ext^1_\MHS(\Q,\Q(2m-1))$.
\end{proof}

\section{Lie Algebra Aspects}

One can also define the Deligne--Beilinson cohomology of a Lie algebra with MHS.
This was done in the unipotent case in \cite{carlson-hain} where it was related
to extensions of unipotent variations of MHS. Here we briefly consider the
general case and its relation to the DB-cohomology of affine groups.

\subsection{Generalities}

Suppose that $\h$ is a Lie algebra over the field $\F$ of characteristic zero.
Denote its enveloping algebra by $U\h$. The cohomology of $\h$ with coefficients
in an $\h$-module $V$ is defined by
$$
H^\dot(\h,V) := \Ext^\dot_{U\h}(\F,V).
$$
It can be computed using the standard projective resolution $U\h \otimes
\Lambda^\dot \h \to \F$, which computes $H^\dot(\h,V)$ from the
Chevalley--Eilenberg cochains
$$
\cC^\dot(\h,V) = \Hom_\F(\Lambda^\dot \h,V) \cong
\Hom_{U\h}(U\h \otimes \Lambda^\dot \h,V)
$$
with its standard differential. (See \cite{cartan-eilenberg} for details.)

If $\h$ is a topological Lie algebra (such as a pronilpotent Lie algebra) which
is topologized as the limit of a projective system of Lie algebras $\h_\alpha$,
we define the continuous cohomology $H^\dot_\cts(\h,V)$ of $\h$ 
by
$$
H^\dot_\cts(\h,V) = \varinjlim_\alpha H^\dot(\h_\alpha,V).
$$
It can be computed using the complex
$$
\cC^\dot_\cts(\h,V) := \varinjlim_\alpha  \cC^\dot(\h_\alpha,V)
$$
of continuous cochains.

\subsection{Hodge theory}

Suppose that $\h$ is a (finite dimensional) Lie algebra in $\MHS_\F$. A  Hodge
representation of $\h$ is a finite dimensional $\h$-module $V$ endowed with an
$\F$-MHS for which the action $\h\otimes V \to V$ is a morphism of MHS. Denote
the category of Hodge representations of $\h$ by $\HRep(\h)$. When $V$ is a
Hodge representation of $\h$, the Chevalley--Eilenberg complex $\cC^\dot(\h,V)$
is a complex in $\MHS_\F$, which implies that $H^\dot(\h,V)$ has a MHS.

More generally, suppose that $\h$ is a topological Lie algebra which is
topologized as the inverse limit of a projective system $\{\h_\alpha\}$ of Lie
algebras in $\MHS_\F$. We will refer to such a Lie algebra as a Lie algebra in
pro-$\MHS_F$. Define $\HRep(\h)$ by
$$
\HRep(\h) = \varinjlim_\alpha \HRep(\h_\alpha).
$$
For each $V$ in $\HRep(\h)$, the continuous cochain complex
$\cC^\dot_\cts(\h,V)$ is a complex in ind-$\MHS_\F$.

\begin{definition}
The Deligne--Beilinson cohomology $H_\cD^\dot(\h,V)$ of a Lie algebra $\h$ in
pro-$\MHS_\F$ and coefficients in an object $V$ of $\HRep(\h)$ is defined to be
the cohomology of the complex
$$
\cone\big(F^0W_0\cC^\dot_\cts(\h,V)_\C \oplus W_0 \cC^\dot_\cts(\h,V)_\F
\to W_0\cC^\dot_\cts(\h,V)_\C\big)[-1].
$$
\end{definition}

Standard arguments imply that, for all $j\ge 0$, there is a short exact sequence
$$
0 \to \Ext^1_\MHS\big(\F,H^{j-1}(\h,V)\big) \to H^j_\cD(\h,V)
\to \Hom_\MHS\big(\F,H^j(\h,V)\big) \to 0.
$$

\subsection{Relation to group cohomology}

Suppose that $G$ is an affine (and therefore proalgebraic) $\F$-group. Its Lie
algebra $\g$ is the inverse limit of a finite dimensional Lie algebras. The
functor $\Rep(G) \to \Rep^\cts(\g)$ induces a map
\begin{equation}
\label{eqn:grp-la}
H^\dot(G,V) \to H^\dot_\cts(\g,V)
\end{equation}
for all $V$ in $\Rep(G)$. This is an isomorphism when $G$ is prounipotent. This
map can be realized on the chain level by choosing an injective resolution $V
\hookrightarrow I^\dot$ of $V$ in $\Rep(G)$. Since the complex $I^\dot$ is a
resolution of $V$ in $\Rep(\g)$, the induced chain map
$$
\cC^\dot(\Hom_\F^\cts(\Lambda^\dot\g,V)) \hookrightarrow
\tot\Hom_\F^\cts(\Lambda^\dot\g,I^\dot)
$$
is a quasi-isomorphism. So the right hand complex computes $H^\dot_\cts(\g,V)$.
The inclusion
$$
(I^\dot)^G \hookrightarrow \tot\Hom_\F^\cts(\Lambda^\dot\g,I^\dot)
$$
is a chain map which induces (\ref{eqn:grp-la}).

Denote the maximal ideal of $\cO(G)$ of functions that vanish at the identity by
$\m_e$. The Lie algebra $\g$ of $G$
is
$$
\g = \Hom_\F(\m_e/\m_e^2,\F).
$$
Its bracket is dual to the skew symmetrization $\m_e/\m_e^2 \to
(\m_e/\m_e^2)^{\otimes 2}$ of the reduced diagonal $\m_e \to \m_e\otimes\m_e$.
This implies that if $G$ has an $\F$-MHS, then its Lie algebra is a Lie algebra
in the category pro-$\MHS_\F$ and that the restriction functor $\Rep(G) \to
\Rep^\cts(\g)$ takes Hodge representations to Hodge representations.

\begin{proposition}
If $G$ is an affine group with an $\F$-MHS, then for all Hodge representations
$V$ of $G$, the natural map (\ref{eqn:grp-la}) is a morphism of MHS and there is
a natural map
$$
H^\dot_\cD(G,V) \to H^\dot_\cD(\g,V)
$$
which is an isomorphism when $G$ is prounipotent.
\end{proposition}

\begin{proof}
Let $V\to I^\dot$ be the standard injective resolution of $V$ in $\Rep(G)$. The
linear map $(I^\dot)^G = V\otimes \cO(G)^{\otimes n} \to
\Hom_\F(\Lambda^{n}\g,V)$ defined by
$$
v\otimes f_1 \otimes \dots \otimes f_n \mapsto
v\otimes (df_1\wedge \dots \wedge df_n)
$$
is a chain map that induces the natural map $H^\dot(G,V) \to H^\dot(\g,V)$. It
is a morphism of MHS, which implies that the natural map is a morphism of MHS.
It also induces a map of Deligne cochains, and therefore a map of Deligne
cohomology.
\end{proof}

When $G$ is not prounipotent, the homomorphism (\ref{eqn:grp-la}) is generally
not an isomorphism. To consider the difference, we adopt the natural hypotheses
of Section~\ref{sec:nat-hypoth}. Namely, $G = W_0 G$ and that $R:=\Gr^W_0 G$ is
reductive. In this case, the restriction map $H^\dot(G,V) \to H^\dot(U,V)$ is an
inclusion of MHS which induces an isomorphism of MHS 
$$
H^\dot(G,V) \overset{\simeq}{\longrightarrow} H^\dot(U,V)^R.
$$
We can thus express $H^\dot_\cD(G,V)$ as an extension
$$
0 \to \Ext^1_\MHS(\F,H^{j-1}(U,V)^R) \to H^j_\cD(G,V) \to \G H^j(U,V)^R \to 0.
$$
This sequence can also be deduced from the Hochschild-Serre spectral sequence
of the extension
$$
1 \to U \to \pi_1(\MHS)\ltimes G \to \pi_1(\MHS)\ltimes R \to 1.
$$
It seems that the group $R$ does not (in general) act on $H^j_\cD(U,V)$, so
it does not make sense to write $H^\dot_\cD(G,V) = H^\dot_\cD(U,V)^R$.

\section{The Eisenstein Projection of the Cup Product}
\label{sec:eis_projn}

Here we give an alternative approach to Brown's computation
\cite[Thm.~10.1]{brown:mmv} of the Eisenstein projection of $\bG_{2n+2m}\cup
\bG_{2m}$. We do this by restricting to the punctured $q$-disk, where the
variations of MHS that correspond to the $\bG_{2k}$ become unipotent. This
allows us to use the Lie algebra version of DB-cohomology explained in the
previous section. While not essential, it does simplify the computation.

The inclusion $\D^\ast \to \M_{1,1}$ of the punctured $q$-disk into $\M_{1,1}$
induces a homomorphism
$$
\pi_1(\D^\ast,\partial/\partial q) \to \pi_1(\M_{1,1}^\an,\partial/\partial q)
\cong \SL(H_\Z)
$$
whose image is the unipotent element $T$ whose logarithm is the nilpotent
endomorphism $N=-\aa\partial/\partial \bb$. The unipotent completion of
$\G_\infty := \pi_1(\D^\ast,\partial/\partial q)$ is a unipotent group
isomorphic to $\Ga$ whose Lie algebra $\g_\infty$ is isomorphic to $\Q(1)$ in
$\MHS$. The inclusion $\G_\infty \hookrightarrow \G$ induces a homomorphism
$\cG_\infty \hookrightarrow \cG$ which is a morphism of MHS.

\begin{remark}
Hodge representations of $\cG_\infty$ correspond to admissible variations of MHS
on the punctured tangent space $T_{e_o}'\Mbar_{1,1}:=T_{e_o}\Mbar_{1,1}-\{0\}$
at the cusp $e_o$. These are nilpotent orbits of MHS.
\end{remark}

\subsection{DB-cohomology of $\cG_\infty$}

The Deligne cohomology of $\cG_\infty$ is particularly simple, which makes it
well suited to computations. If $V$ is a $\cG_\infty$ module, then
$$
H^0(\cG_\infty,V) = V^N,\quad H^1(\cG^\infty, V) = (V/NV)(-1),
$$
and all other cohomology groups vanish. If $V$ is an $\SL(H)$-module, viewed as
a $\cG_\infty$-module via the homomorphism, $\cG_\infty \to \cG \to \SL(H)$,
then $V/NV$ is isomorphic to the space of highest weight vectors in $V$.

If $V$ is a Hodge representation of $\cG_\infty$, then there is a natural
isomorphism
$$
H^2_\cD(\cG_\infty,V(r)) \cong \Ext^1_\MHS\big(\Q,(V/NV)(r-1)\big)
\cong V^\DR_\C/\big((2\pi i)^{r-1} V_\Q^B + NV_\C^\DR\big).
$$

\subsection{Lie algebra cochains}

Denote the coordinate in $T_{e_o}\Mbar_{1,1}$ with respect to the basis vector
$\partial/\partial q$ by $q_o$. The coordinate ring of $\cG_\infty^\DR$ is the
polynomial algebra over $\Q$ generated by the iterated integral $[dq_o/q_o]$.
Since
$$
\overbrace{[dq_o/q_o|\dots|dq_o/q_o]}^r = [dq_o/q_o]^r/r!,
$$
the cotangent space of $\cG^\infty$ is spanned by the coset of $[dq_o/q_o]$ mod
iterated integrals of $dq_o/q_o$ of length $\ge 2$. Let $\lambda_\DR$ be the
linear function on the tangent space $\g_\infty^\DR$ of $\cG^\DR_\infty$ at $1$
that takes the value 1 on $[dq_o/q_o]$. Let $\lambda_B = \lambda_\DR/2\pi i$. It
is a linear functional on $\g_\infty^B$.

Suppose that $V$ is a Hodge representation of $\g_\infty$. Then
$$
\cC^\dot(\g_\infty^\DR,V^\DR) = \Lambda^\dot(d\lambda_\DR)\otimes V^\DR
$$
with the differential $\delta : V^\DR \to V^\DR\otimes d\lambda^\DR$ defined by
$v \mapsto N v\, d\lambda_\DR$. The Betti complex $\cC^\dot(\g_\infty^B,V^B)$ is
defined similarly. It has differential $\delta : V^B \to V^B\otimes d\lambda^B$
defined by $\quad v \mapsto N^B v\, d\lambda_B$, where $N_B = N/2\pi i$. When $V
= S^m H$, $N = \aa\partial/\partial \bw$ and $N_B = -\aa\partial/\partial \bb$.

\subsection{Computation of the Eisenstein projection}

Denote by $\bG^\infty_{2k}$ the image of $\bG_{2k}$ under the restriction map
$$
H^1_\cD\big(\cG,S^{2k-2}H(2k-1)\big) \to
H^1_\cD\big(\cG_\infty,S^{2k-2}H(2k-1)\big).
$$

\begin{proposition}
\label{prop:cup_infty}
If $n > 0$ and $m > 1$, then
$$
\bG_{2n+2m}^\infty \cup \bG_{2m}^\infty \in
H^2_\cD\big(\cG_\infty,S^{2n+2m-2}H\otimes S^{2m-2}H(2n+4m-2)\big)
$$
is the coset of
$$
-\zeta(2m-1)\frac{(2m-2)!}{2}\frac{B_{2n+2m}}{4n+4m}
\bw^{2n+2m-2}\otimes\aa^{2m-2}
$$
in
\begin{multline*}
\Ext^1_\MHS\Big(\Q,
H^1(\cG^\infty,S^{2n+2m-2}H\otimes S^{2m-2}H(2n+4m-2)\big)\Big) \cr
\cong \Ext^1_\MHS\big(\Q,[S^{2n+2m-2}H\otimes S^{2m-2}H(2n+4m-3)]/\im N\big)
\cr
\cong \bigoplus_{r=0}^{2m-2}\Ext^1_\MHS(\Q,\Q(r+1)). \hspace*{2.25in}
\end{multline*}
Its projection to 
$$
\Ext^1_\MHS\Big(\Q,H^1\big(\cG^\infty,S^{2n+4m-4-2r}H(2n+4m-2-r)\big)\Big)
\cong \Ext^1_\MHS(\Q,\Q(r+1))
$$
vanishes when $0\le r< 2m-2$.
\end{proposition}

To compute the projection of $\bG_{2n+2m}^\infty\cup \bG_{2m}^\infty$ to
$H^2_\cD(\cG_\infty,S^{2n}H(2n+2m))$, we need to specify an $\SL(H)$-invariant
projection $S^{2n+2m-2}H\otimes S^{2m-2}H \to S^{2n}H$. To this end, define
$$
\partial : S^aH(a)\otimes S^b H(b) \to S^{a-1}H\otimes S^{b-1}H(a+b-1)
$$
to be the operator
$$
\partial :=
\partial/\partial\bw\otimes \partial/\partial \aa
- \partial/\partial\aa\otimes \partial/\partial \bw
= (2\pi i)^{-1}
(\partial/\partial\aa\otimes \partial/\partial \bb
- \partial/\partial\bb\otimes \partial/\partial \aa).
$$
It is $\SL(H)$ equivariant and preserves the $\Q$-DR and $\Q$-Betti structures.

Identify $S^{2n}H$ with $S^{2n}H\otimes 1$. Then $\partial^{2m-2}$ is an
$\SL(H)$-invariant operator
$$
\partial^{2m-2} : S^{2n+2m-2}H\otimes S^{2m-2}H(2n+4m-2) \to S^{2n}H(2n+2m).
$$
Since
$$
\partial^{2m-2}(\bw^{2n+2m-2}\otimes\aa^{2m-2})
= (2m-2)!\frac{(2n+2m-2)!}{(2n)!} \bw^{2n}\otimes 1,
$$
and since the Eisenstein projection (\ref{eqn:projnsQ}) is just the restriction
mapping
\begin{multline*}
H^2_\cD(\cG,S^{2n}H(r)) \to H^2_\cD(\cG_\infty,S^{2n}H(r)) 
\cr
\cong H^2_\cD(\cG_\infty,S^{2n}H(r)) \cong \Ext^1_\MHS(\Q,\Q(2n+r)),
\end{multline*}
we have:

\begin{corollary}
If $n > 0$ and $m > 1$, then
$$
\partial^{2m-2} (\bG_{2n+2m}\cup \bG_{2m})
\in H^2_\cD(\cG_\infty,S^{2n}H(2n+2m))
$$
is the coset of
$$
-\zeta(2m-1)\frac{[(2m-2)!]^2}{2}
\frac{(2n+2m-2)!}{(2n)!}\frac{B_{2n+2m}}{4n+4m}
\bw^{2n}
$$
in
\begin{multline*}
\Ext^1_\MHS\Big(\Q,
H^1(\cG^\infty,S^{2n}H(2n+2m)\big)\Big) \cr
\cong \Ext^1_\MHS\big(\Q,[S^{2n}H(2n+2m-1)]/\im N\big)
\cong \Ext^1_\MHS(\Q,\Q(2m-1)). \qed
\end{multline*}
\end{corollary}

\begin{proof}[Proof of Proposition~\ref{prop:cup_infty}]

The restriction of the cocycle $\bG_{2k}$ to $\cG^\infty$ is
$$
\begin{bmatrix}
\xymatrix@C=-24pt@R=0pt{
& -\frac{(2k-2)!}{2}\zeta(2k-1)\aa^{2k-2} \cr
-\frac{B_{2k}}{4k}\bw^{2k-2}[dq_o/q_o] &&
-\frac{B_{2k}}{4k}(2\pi i\bb)^{2k-2} [dq_o/q_o]
}
\end{bmatrix}.
$$
Consequently, the corresponding 1-cocycle on $\g_\infty$ is
$$
Z_{2k}^\infty :=
\begin{bmatrix}
\xymatrix@C=-24pt@R=0pt{
& -\frac{(2k-2)!}{2}\zeta(2k-1)\aa^{2k-2} \cr
-\frac{B_{2k}}{4k}\bw^{2k-2}d\lambda_\DR &&
-2\pi i\frac{B_{2k}}{4k}(2\pi i\bb)^{2k-2} d\lambda_B
}
\end{bmatrix}.
$$
The class of $\bG_{2n+2m}\cup\bG_{2m}$ is represented by
$$
Z_{2n+2m}^\infty \cup_0 Z_{2m}^\infty
=
\begin{bmatrix}
& C \cr
0 && 0
\end{bmatrix}
$$
where
\begin{multline*}
C = \bigg(
\frac{(2n+2m-2)!}{2}\frac{B_{2m}}{4m}\zeta(2n+2m-1)\,
\aa^{2n+2m-2}\otimes\bw^{2m-2}
\cr
- \frac{(2m-2)!}{2}\frac{B_{2n+2m}}{4n+4m}\zeta(2m-1)\,
\bw^{2n+2m-2}\otimes\aa^{2m-2}
\bigg)\,d\lambda_\DR.
\end{multline*}
Since the set of $\sl(H)$ weights of highest weight vectors in
$S^{2n+2m-2}H\otimes S^{2m-2}$ is $\{2n,2n+2,\dots,2n+2m-2\}$, and since
$\aa^{2n+2m-2}\otimes\bw^{2m-2}$ has weight $-2n$, it is in the image of $N$. So
the class of $C$ in $H^1(\cG_\infty,S^{2n+2m-2}H\otimes S^{2m-2})$ is also
represented by
$$
C'= - \frac{(2m-2)!}{2}\frac{B_{2n+2m}}{4n+4m}\zeta(2m-1)\,
\bw^{2n+2m-2}\otimes\aa^{2m-2}\,d\lambda_\DR.
$$
Since this vector has $\sl(H)$-weight $2n$, its projection to $S^{2n+4m-4-2r}H$
is trivial, except possibly when $r=2m-2$.
\end{proof}

\subsection{Comparison with $\bZ_{2m-1}\cup \bG_{2n+2}^\infty$}

The classes $\bG_{2n+2m}^\infty\cup \bG_{2m}^\infty$ and $\bZ_{2m-1}\cup
\bG_{2n+2}^\infty$ both live in $H^2_\cD(\cG_\infty,S^{2n}H(2n+2m))$. In order
to compare them, we will use the $\SL(H)$-invariant embedding
$$
D^\ell : S^k H \to S^{k+\ell}H\otimes S^\ell H(\ell)
$$
defined by
$$
f(\aa,\bw) \mapsto
(\bw\otimes\aa-\aa\otimes\bw)^{\ell} f(\aa,\bw)\otimes 1,
$$
which preserves the $\Q$-DR and $\Q$-Betti structures. The case of interest
to us is when $\ell = 2m-2$ and $k=2n$.

\begin{lemma}
\label{lemma:h_wt}
For all $k,\ell\ge 0$, we have
$$
D^\ell(\bw^k\otimes 1) \equiv  %% (-1)^\ell
\frac{k+\ell+1}{k+1}\, \bw^{k+\ell}\otimes \aa^{\ell} \bmod \im N.
$$
\end{lemma}

\begin{proof}
Identify $S^{k+\ell}H$ with polynomials in $x$ of degree $\le k+\ell$ and
$S^\ell H$ with polynomials in $y$ of degree $\le \ell$, where $x^{k+\ell}$
corresponds to the highest weight vector of $S^{k+\ell}H$ and $y^\ell$
corresponds to the highest weight vector of $S^\ell H$. With these conventions,
$D$ corresponds to multiplication by $x-y$, the vector $D^\ell \bw^k$
corresponds to $(x-y)^\ell x^k$, and the nilpotent endomorphism $N$ of
$S^{k+\ell}H\otimes S^\ell H$ corresponds to the derivation
${\partial}/{\partial x} + {\partial}/{\partial y}$.
Since
$$
\bigg(\frac{\partial}{\partial x} + \frac{\partial}{\partial y}\bigg)
\sum_{i=1}^j \frac{(-1)^i}{i}\binom{k+\ell}{i-1}x^{k+\ell-i+1}y^i
= -x^{k+\ell} + (-1)^j\binom{k+\ell}{j} x^{k+\ell-j}y^j,
$$
$$
x^{k+\ell} \equiv (-1)^{j}\binom{k+\ell}{j} x^{k+\ell-j}y^j \bmod \im N.
$$
So
$$
(x-y)^\ell x^k = \sum_{j=0}^\ell(-1)^j \binom{\ell}{j} x^{k+\ell-j}y^j 
\equiv x^{k+\ell}\sum_{j=0}^\ell \binom{\ell}{j}\bigg\slash \binom{k+\ell}{j}
\bmod \im N.
$$
The result now follows from the identity
$$
\sum_{j=0}^\ell \binom{\ell}{j}\bigg\slash \binom{k+\ell}{j}
= \frac{k+\ell+1}{k+1},
$$
which can be proved by showing that the both sides satisfy the recursion
$$
(k+\ell)F(k,\ell) = kF(k-1,\ell) + \ell F(k,\ell-1)
$$
and verifying equality when either $k$ or $\ell$ is zero. It can also be
proved using the binomial transformation.
\end{proof}

\begin{theorem}
If $n > 0$ and $m > 1$, then
$$
\bG_{2n+2m}^\infty\cup \bG_{2m}^\infty
=
\frac{(2m-2)!}{2}\frac{\binom{2n+2}{2}}{\binom{2n+2m}{2}}
\frac{B_{2n+2m}}{B_{2n+2}}D^{2m-2}\big(\bZ_{2m-1}\cup \bG_{2n+2}\big)
$$
in $H^2_\cD\big(\cG,S^{2n+2m-2}H\otimes S^{2m-2}(2n+2m)\big)\cong
\Ext^1_\MHS(\Q,\Q(2m-1))$.
\end{theorem}

\begin{proof}
Proposition~\ref{prop:external_cup} implies that
$\bZ_{2m-1}\cup\bG_{2n+2}^\infty$ is represented by the coset of
$$
- \zeta(2m-1) \,\frac{B_{2n+2}}{4n+4}\, \bw^{2n}\otimes 1
$$
in
$$
\Ext^1_\MHS\big(\Q,H^1(\cG_\infty,S^{2n}H(2n+2m))\big)
\cong \Ext^1_\MHS\big(\Q,S^{2n}H(2n+2m-1)/\im N\big)
$$
Proposition~\ref{prop:cup_infty} and Lemma~\ref{lemma:h_wt} imply that
$\bG_{2n+2m}^\infty\cup \bG_{2m}^\infty$ is represented by
$$
\frac{(2m-2)!}{2}\frac{\binom{2n+2}{2}}{\binom{2n+2m}{2}}
\frac{B_{2n+2m}}{B_{2n+2}}
\bigg[- \zeta(2m-1) \,\frac{B_{2n+2}}{4n+4}\, D^{2m-2}(\bw^{2n}\otimes 1)\bigg].
$$
\end{proof}

\appendix

\section{The Orbifold Case}

For applications to (say) modular curves, one needs to know that
Theorem~\ref{thm:vmhs} holds for orbifolds that are globally the quotient of a
smooth variety by a finite group. Here we sketch a proof that this is indeed the
case.

Suppose that $X$ is the orbifold quotient of a smooth connected variety $Y$ by a
finite group $\G$, which acts on $Y$ on the right. We write $X=Y\fss\G$. Note
that we are not assuming that $\G \to \Aut Y$ is injective.

To work on $X$ is to work with $\G$-equivariant objects on $Y$. For example, a
local system (resp.\ admissible VMHS) over $X$ is (by definition) a local system
(resp.\ admissible VMHS) over $Y$ with a right $\G$ action that commutes with
the projection to $X$.

The fundamental group $\pi_1(X,y_o)$, where $y_o\in Y$ is a ``geometric point''
of $X$, is the group whose elements are pairs $(\alpha,\gamma)$, where $\alpha :
[0,1]\to Y$ is a homotopy class of paths from $y_o$ to $y_o\gamma$.
Multiplication is
$$
(\alpha,\gamma)\circ(\beta,\mu) = (\alpha|_\mu\circ \beta,\gamma\mu),
$$
where $\circ$ denotes path composition (in the functional order) and
$\alpha|_\mu$ denotes the translate of $\alpha$ by $\mu$.\footnote{This
is equivalent to the usual definition of the orbifold fundamental group of $X$.
The description given here is the same as the one given in \cite{hain:elliptic},
except that it is modified for the change in the path multiplication
convention.} The projection $(\alpha,\gamma) \mapsto \gamma$ defines a 
homomorphism $\pi_1(X,y) \to \G$. There is an exact sequence
$$
1 \to \pi_1(Y,y_o) \to \pi_1(X,y_o) \to \G \to 1.
$$
If $\V$ is a local system over $X$, then one has a monodromy representation
$$
\pi_1(X,y_o) \to \Aut V_o
$$
where $V_o$ denotes the fiber over $y_o$. If $\H$ is a PVHS over $X$, then the
coordinate ring of the Zariski closure $R_o$ of the monodromy representation
$\pi_1(X,y_o) \to \Aut H_o$ has a Hodge structure of weight $0$.

The construction of the MHS on $\cO(\cG_X)$ is similar to that in the
non-orbifold case. First recall that the $C^\infty$ de~Rham complex of $X$ with
coefficients in a (real or complex) local system $\V$ over $X$ is defined to be
the complex of $\G$-invariant forms on $Y$ with values in $\V$:
$$
E^\dot(X,\V) := E^\dot(Y,\V)^\G.
$$
One uses the construction from \cite{hain:malcev} using the
complex
$$
E^\dot_\fin(X,\bO_o) := E^\dot_\fin(Y,\bO_o)^\G
$$
of $R_o$-finite vectors,\footnote{These are the elements of $E^\dot(Y,\bO_o)^\G$
that lie in a finite dimensional $R_o$-submodule.} where $\bO_o$ is the local
system over $X$ whose fiber over $y_o$ is $\cO(R_o)$ and whose monodromy
representation factors through right multiplication by $R_o$.

To show that the orbifold version of Theorem~\ref{thm:vmhs} follows from the
standard version, Theorem~\ref{thm:vmhs}, we use a standard trick. Fix a simply
connected smooth variety $W$ on which $\G$ acts fixed point freely.\footnote{For
example, one can take $W$ to be $(\C^N)^\G - \Delta$, where $\Delta$ is the fat
diagonal and $N>1$. This is $N-1$ connected. The group $\G$ acts by permuting
the coordinates.} Fix $w_o\in W$ and set $z_o=(y_o,w_o)$. The projection
$Y\times W \to Y$ induces an isomorphism on fundamental groups. Set
$$
Z = (Y\times W)/\G.
$$
This is a smooth variety. The projection $Y\times W \to Y$ induces an orbifold
map $(Z,z_o) \to (X,x_o)$ and an isomorphism $\pi_1(Z,z_o) \to \pi_1(X,y_o)$. So
the relative completion of $\pi_1(Z,z_o)$ is isomorphic to the relative
completion of $\pi_1(X,y_o)$. Their MHSs are also isomorphic as the projection
$$
E^\dot_\fin(Z,\bO_o) = E^\dot_\fin(Y\times W,\bO_o)^\G \to
E^\dot_\fin(Y,\bO_o)^\G = E^\dot_\fin(X,\bO_o)
$$
induced by the $\G$-invariant inclusion $Y\times w_o \hookrightarrow Y\times W$
is a morphism of mixed Hodge complexes which induces an isomorphism on homology
in degrees $<2$ and an injection in degree $2$.

Theorem~\ref{thm:vmhs} implies that $\HRep(\cG_X)$ is equivalent to
$\MHS(Z,\H)$. Variations of MHS over $Z$ are equivalent to $\G$-invariant
variations over $Y\times W$. Since $W$ is simply connected, variations of MHS
over $Z$ are constant when pulled back to any $W$-slice $y\times W \to Z$ and
are therefore determined by their pullback to the $Y$-slice $Y\times w_o$. In
this way, $\G$-invariant variations of MHS over $Y\times W$ correspond to
$\G$-invariant variations over $Y\cong Y\times w_o$.

\end{document}